\documentclass[10pt]{amsart}

\usepackage{amsmath,amsthm,amssymb,amsaddr}
\usepackage{enumerate,tikz,hyperref}

\newcommand{\class}{\mathsf}

\newcommand{\alg}{\mathbf}
\newcommand{\structure}[1]{\mathbb{#1}}

\newcommand{\set}[2]{\{ #1 \mid #2 \}}
\newcommand{\pair}[2]{\langle #1, #2 \rangle}
\newcommand{\tuple}[1]{\overline{#1}}
\newcommand{\range}[1]{[#1]}
\newcommand{\card}[1]{| #1 |}

\newcommand{\assign}{\mathrel{:=}}
\newcommand{\bsubseteq}[1]{\subseteq_{#1}}
\newcommand{\corightarrow}{\mathbin{-}}
\newcommand{\equals}{\approx}
\newcommand{\idmap}{\mathrm{id}}

\newcommand{\into}{\hookrightarrow}

\newcommand{\True}{\mathsf{1}}
\newcommand{\False}{\mathsf{0}}

\newcommand{\Eps}{\mathrm{E}}
\newcommand{\FmAlg}{\alg{Tm}}
\newcommand{\Fsymbol}{\mathrm{F}}

\DeclareMathOperator{\Fg}{Fg}
\DeclareMathOperator{\Fi}{Fi}

\newcommand{\BA}[1]{\alg{B}_{#1}}
\newcommand{\dBA}[2]{\alg{B}_{#1 \times #2}}
\newcommand{\BAm}[1]{\structure{B}_{#1}}
\newcommand{\nonempty}[1]{P_{#1}}
\newcommand{\dnonempty}[2]{P_{#1,#2}}
\newcommand{\atom}[1]{\mathsf{a}_{#1}}
\newcommand{\coatom}[1]{\mathsf{c}_{#1}}
\newcommand{\fg}[2]{[#2]_{#1}}
\newcommand{\height}[2]{F_{#1}^{#2}}

\newcommand{\Mfive}{\alg{M}_{\alg{5}}}
\newcommand{\Nfive}{\alg{N}_{\alg{5}}}

\newcommand{\FreeDLat}[1]{\alg{F}_{\mathsf{DL}}(#1)}
\newcommand{\FreeBA}[1]{\alg{F}_{\mathsf{BA}}(#1)}

\newcommand{\FreeuSLat}[1]{\alg{F}_{\mathsf{uSL}}(#1)}

\newcommand{\DLclass}[1]{\class{DL}_{#1}}
\newcommand{\BAclass}[1]{\class{BA}_{#1}}
\newcommand{\SLclass}[1]{\class{SL}_{#1}}
\newcommand{\uSLclass}[1]{\class{uSL}_{#1}}

\newcommand{\BAlg}{\class{BA}}
\newcommand{\DLat}{\class{DLat}}
\newcommand{\SLat}{\class{SLat}}
\newcommand{\uSLat}{\class{uSLat}}

\newcommand{\HSop}{\mathbb{H}_{\mathrm{S}}}
\newcommand{\HSinvop}{\mathbb{H}^{-1}_{\mathrm{S}}}
\newcommand{\Sop}{\mathbb{S}}
\newcommand{\Pop}{\mathbb{P}}
\newcommand{\PUop}{\mathbb{P}_{\mathrm{U}}}

\newtheorem*{theorem*}{Theorem}
\newtheorem{theorem}{Theorem}[section]
\newtheorem{lemma}[theorem]{Lemma}

\newtheorem{corollary}[theorem]{Corollary}

\newtheorem{fact}[theorem]{Fact}

\author{Adam P\v{r}enosil}
\address{Universit\`{a} degli Studi di Cagliari, Cagliari, Italy}
\email{adam.prenosil@unica.it}
\title{Filter classes of upsets of distributive lattices}
\date{}
\thanks{The author gratefully acknowledges the support of Fondazione di Sardegna within the project ``Resource sensitive reasoning and logic'', Cagliari, CUP: F72F20000410007. The author also wishes to thank the two anonymous referees for their useful comments, which helped to improve the clarity of the manuscript.}

\begin{document}

\begin{abstract}
  Let us say that a class of upward closed sets (upsets) of distributive lattices is a \mbox{finitary} filter class if it is closed under homomorphic preimages, intersections, and directed unions. We show that the only finitary filter classes of up\-sets of distributive lattices are formed by what we call $n$-filters. These are related to the finite Boolean lattice with $n$ atoms in the same way that filters are related to the two-element Boolean \mbox{lattice}: $n$-filters are precisely the intersections of prime $n$-filters and prime $n$-filters are precisely the homo\-morphic pre\-images of the prime $n$-filter of non-zero elements of the finite Boolean lattice with $n$ atoms. Moreover, $n$-filters on Boolean algebras are the only finitary filter classes of upsets of Boolean algebras generated by prime \mbox{upsets}.
\end{abstract}

  \maketitle

\section{Introduction}

  The structure theory of distributive lattices rests on two facts: the prime filter separation property (which states that each filter is an intersection of prime filters) and the correspondence \mbox{between} prime filters and homo\-morphisms into the two-element distributive lattice $\BA{1}$ consisting of the elements $\False < \True$ (which states that each prime filter is a homomorphic preimage of the prime filter $\{ \True \}$ on $\BA{1}$). Together, these facts show that distributive lattices are subdirect powers of~$\BA{1}$.

  Are there other kinds of upward closed subsets (upsets) of distributive lattices besides filters which admit an analogous representation in terms of intersections of homomorphic preimages of a certain canonical upset? We show that if the role of the canonical prime filter $\{ \True \}$ on $\BA{1}$ is taken over by the upset $\nonempty{n}$ of non-zero ($x > \False$) elements of the finite Boolean lattice with $n$ atoms ${\BA{n} \assign (\BA{1})^{n}}$, we obtain precisely what we call \emph{$n$-filters}. An $n$-filter is an upset $F$ of a meet semilattice $\alg{S}$ which enjoys the following property for each non-empty finite set $X \subseteq F$:
\begin{align*}
  \text{if $\bigwedge Y \in F$ for each $Y \subseteq X$ with $1 \leq \card{Y} \leq n$, then $\bigwedge X \in F$.}
\end{align*}
  Each $m$-filter is an $n$-filter if $m < n$, but the $n$-filter $\nonempty{n}$ is not an $m$-filter for any $m < n$. Clearly $1$-filters in this sense are just ordinary lattice filters.  An $n$-filter is said to be \emph{prime} if $a \vee b \in F$ implies that $a \in F$ or $b \in F$.

  The relationship between prime $n$-filters on distributive lattices and the canonical prime $n$-filter $\nonempty{n}$ on $\BA{n}$ extends the relationship between prime filters and the canonical prime filter $\{ \True \}$ on $\BA{1}$:
\begin{align*}
  \text{prime $n$-filters are precisely the homomorphic preimages of $\nonempty{n}$.}
\end{align*}
  This relationship follows from the observation that prime $n$-filters on distributive lattices are precisely the unions of at most $n$ ordinary prime filters (which is far from true for $n$-filters in general).

  The relationship between $n$-filters and prime $n$-filters on distributive lattices then extends the relationship between filters and prime filters:
\begin{align*}
  \text{$n$-filters are precisely the intersections of prime $n$-filters.}
\end{align*}
  The key part of the proof is describing the $n$-filter generated by an upset of a distributive lattice. Observe that if we apply this description of $n$-filters to ideal lattices of distributive lattices, we obtain the filter--ideal separation property for distributive lattices: 
\begin{align*}
  \text{$n$-filters can be separated from ideals by prime $n$-filters.}
\end{align*}
  In other words, the fundamental properties of (prime) filters on distributive lattices have direct analogues for (prime) $n$-filters.

  Are there other kinds of upsets of distributive lattices besides $n$-filters which share the basic properties of lattice filters? We encapsulate these properties in the definition of a \emph{finitary filter class}: a class of upsets closed under homo\-morphic preimages, intersections, and directed unions. The only finitary filter classes of upsets of distributive lattices are indeed the families of $n$-filters for some $n$. This is not true for Boolean algebras, where many other finitary filter classes exist. Nonetheless, $n$-filters are the only finitary filter classes of upsets of Boolean algebras generated by some \mbox{family} of prime upsets. There are also other finitary filter classes of upsets of unital meet semilattices. However, $n$-filters are the only finitary filter classes of non-empty upsets of unital meet semilattices which are closed under so-called strict homomorphic images. This is not true for general meet semilattices.

  To the best of our knowledge, the problem of describing which families of upsets of distributive lattices (and other kinds of ordered algebras) behave like the family of all lattice filters has not been considered in the literature. Equivalently, as explained in more detail in Section~\ref{sec: filter classes}, this problem asks for a full classification of upsets according to which first-order properties of a particular type they satisfy (namely properties expressible by what we call filter implications and logical implications in Theorem~\ref{thm: axiomatization by implications}). In the case of distributive lattices and unital meet semilattices, we provide a full classification of their upsets according to which logical implications they satisfy (Theorems~\ref{cor: extensions of dlinfty} and~\ref{cor: extensions of uslinfty}). The analogous problem for Boolean algebras is substantially more complicated. Indeed, the major questions left open by the present paper are whether there are tractable classifications of upsets of Boolean algebras and meet semilattices. (We conjecture that the answer is negative for Boolean algebras and positive for meet semilattices.)

  The conceptual machinery used in the later sections (from Section~\ref{sec: filter classes} onwards) to solve these problems will largely be familiar to practitioners of so-called abstract algebraic logic~\cite{font16}, but perhaps not to lattice theorists. This might require the reader to first get to grips with some possibly unfamiliar jargon and notation. We have tried to keep this to the bare minimum required to prove the desired results.

  The concept of an $n$-filter may remind the reader of the so-called \emph{$n$-Helly property} for families of sets. A family of sets enjoys this property if for each non-empty finite subfamily $S_{i}$ with $i \in I$
\begin{align*}
  \bigcap_{i \in I} S_{i} = \emptyset & \implies \bigcap_{j \in J} S_{j} = \emptyset \text{ for some } J \subseteq I \text{ with } 1 \leq \card{J} \leq n.
\end{align*}
  This property is named after Helly's theorem in convex geometry, which states that the family of all convex subsets of $\mathbb{R}^{n}$ enjoys the $(n+1)$-Helly property~\cite{helly23}. In~the terminology of the present paper we would say that the non-empty convex subsets of $\mathbb{R}^{n}$ form an $(n+1)$-filter in the lattice of all convex subsets of $\mathbb{R}^{n}$. The characteristic equational property of this lattice is the $n$-distributive property introduced by Huhn~\cite{huhn83}:
\begin{align*}
  & x \wedge (y_{1} \vee \dots \vee y_{n+1}) = (x \wedge z_{1}) \vee \dots \vee (x \wedge z_{n+1}),
\end{align*}
  where
\begin{align*}
 & z_{i} \assign \bigwedge \set{y_{j}}{1 \leq j \leq n+1 \text{ and } j \neq i}.
\end{align*}
  While $n$-distributivity will not make any further appearance in the present paper, it is worth observing that this property is intimately related to $n$-filters: a lattice is $n$-distributive if and only if its filters can be separated from its ideals by $n$-prime filters (i.e.\ the complements of prime $n$-ideals). This separation property is, in a way, dual to the one considered in this paper, which states that $n$-filters can be separated from ideals by prime $n$-filters.

\section{Basic properties of \texorpdfstring{$n$-filters}{n-filters}}

  We first introduce $n$-filters, prime $n$-filters, and $m$-prime filters in more detail. These may be defined on any meet semilattice (in fact on any poset), but restricting to distributive meet semi\-lattices gives us a better grip on the generation of $n$-filters.\footnote{The $n$-filters studied here are unrelated to the $n$-filters of {\L}ukasiewicz--Moisil algebras~\cite{busneag+chirtes07}.}

\subsection{Defining \texorpdfstring{$n$-filters}{n-filters}}

  It~will be convenient to introduce the notation ${Y \bsubseteq{n} X}$ to abbreviate the claim that~$Y$ is a non-empty subset of $X$ of cardinality at~most~$n$. The~set $\{ 1, \dots, n \}$ will be denoted~$\range{n}$. In other words,
\begin{align*}
  Y \bsubseteq{n} X \iff Y \subseteq X \text{ and } \card{Y} \in \range{n}.
\end{align*}
  Throughout the paper, $\alg{S}$ will denote a meet semilattice and $\alg{L}$ will denote a lattice. By a semilattice we shall mean a \emph{meet} semilattice by default. The top (bottom) element of a meet semilattice, if it exists, will be denoted by $\True$ ($\False$), but in general we do not require that these elements exist. An upset $F$ of $\alg{S}$ is called \emph{total} if $F = \alg{S}$. An upset which is not total will be called \emph{proper}.
 
  An \emph{$n$-filter} on a semilattice $\alg{S}$ for $n \geq 1$ is an upset $F \subseteq \alg{S}$ such that for each non-empty finite $X \subseteq \alg{S}$
\begin{align*}
  \bigwedge Y \in F \text{ for each } Y \bsubseteq{n} X \implies \bigwedge X \in F.
\end{align*}
  We further define a \emph{$0$-filter} to be an upset which is either empty or total. An~\emph{\mbox{$n$-ideal}} on a join semilattice is an $n$-filter on the order dual of the join semilattice. A $1$-filter ($1$-ideal) will simply be called a \emph{filter} (\emph{ideal}).

  We may in fact restrict to $\card{X} = n+1$ and $\card{Y} = n$ in the definition of $n$-filters. For example, the~$2$-filters on $\alg{S}$ are precisely the upsets $F \subseteq \alg{S}$ such that
\begin{align*}
  x \wedge y, y \wedge z, z \wedge x \in F \implies x \wedge y \wedge z \in F.
\end{align*}

\begin{fact} \label{fact: restricted definition of n filters}
  An~upset $F$ of a semilattice $\alg{S}$ is an $n$-filter if and only if 
\begin{align*}
  \bigwedge_{j \neq i} x_{j} \in F \text{ for each } i \in \range{n+1} \implies x_1 \wedge \dots \wedge x_{n+1} \in F.
\end{align*}
\end{fact}

\begin{proof}
  The left-to-right implication is trivial. Conversely, suppose that the above implication holds for~$F$. We~prove by induction over the cardinality of the set ${X \subseteq \alg{S}}$ that if $\bigwedge Y \in F$ for each $Y \bsubseteq{n} X$, then $\bigwedge X \in F$. If~$\card{X} \leq n$, we may take ${Y \assign X}$. If~$\card{X} = n+1$, the claim holds by assumption. Now consider $X = \{ x_{1}, \dots, x_{k+1} \}$ with $\card{X} = k+1 > n+1$ and suppose that the claim holds for all $Y \bsubseteq{k} X$. If $\bigwedge X \notin F$, then by the inductive hypothesis applied to $\{ x_{1}, \dots, x_{k-1}, x_{k} \wedge x_{k+1} \}$ there is either some $Y \bsubseteq{n} \{ x_{1}, \dots, x_{k-1} \}$ such that $\bigwedge Y \notin F$ or there is some $Y \bsubseteq{n-1} \{ x_{1}, \dots, x_{k-1} \}$ such that $\bigwedge Y \wedge x_{k} \wedge x_{k+1} \notin F$.  This is a meet of $n+1$ elements and $n+1 \leq k$, therefore the inductive hypothesis applies to $Y \cup \{ x_{k}, x_{k+1} \}$. In either case we obtain some $Z \bsubseteq{n} X$ such that $\bigwedge Z \notin F$.
\end{proof}

  The family of all $n$-filters on a semilattice $\alg{S}$ ordered by inclusion forms a complete lattice $\Fi_{n} \alg{S}$ where meets are intersections of $n$-filters and directed joins are directed unions of $n$-filters. We~can thus define the $n$-filter \emph{generated} by $X \subseteq \alg{S}$ as the smallest \mbox{$n$-filter} on $\alg{S}$ which contains~$X$. An $n$-filter generated by a finite set will be called \emph{finitely generated}. Because each $n$-filter is a directed union of its finitely generated $n$-subfilters, $\Fi_{n} \alg{S}$ is an algebraic lattice and its compact elements are precisely the finitely generated $n$-filters.

  An element $x$ of a meet semilattice $\alg{S}$ will be called \emph{(meet) $m$-prime}, where $1 \leq m$, if for each non-empty finite family $Y \subseteq \alg{S}$
\begin{align*}
  \bigwedge Y \leq x & \implies \bigwedge Z \leq x \text{ for some } Z \bsubseteq{m} Y.
\end{align*}
  We may again restrict to $\card{I} = m+1$ and $\card{J} = m$. We now define an \emph{$m$-prime} $n$-filter on a semilattice $\alg{S}$ as a meet $m$-prime element of the lattice $\Fi_{n} \alg{S}$ of $n$-filters on $\alg{S}$. An \emph{\mbox{$m$-prime}} \mbox{$n$-ideal} on a join semilattice is then an $m$-prime $n$-filter on its order dual. A~\mbox{$1$-prime} $n$-filter will simply be called a \emph{prime} $n$-filter. This agrees with the existing definition of prime filters on semilattices.

  Note that this is a second-order definition, in the sense that it quantifies over subsets of~$\alg{S}$. However, if either $n = 1$ or $m = 1$, it may be reduced to a first-order definition, which quantifies only over elements of $\alg{S}$.\footnote{One can more generally define an $n$-filter on a poset as an upset $F$ such that for each non-empty finite $X \subseteq F$ the set $X$ has a lower bound in $F$ whenever each $Y \bsubseteq{n} {X}$ does. Although the family of all $n$-filters on a poset need not be closed under intersections, we can still call an $n$-filter $F$ on a poset $m$-prime if for each non-empty finite family $F_{i}$ of $n$-filters for $i \in I$
\begin{align*}
  \bigcap_{i \in I} F_{i} \subseteq F \implies \bigcap_{j \in J} F_{j} \subseteq F \text{ for some } J \bsubseteq{m} I.
\end{align*}
  With this definition in hand, Fact~\ref{fact: prime and m-prime} in fact holds for arbitrary posets.}

\begin{fact} \label{fact: prime and m-prime}
  An $n$-filter on a lattice is prime if and only if its complement is an ideal. A filter on a lattice is $m$-prime if and only if its complement is an $m$-ideal.
\end{fact}

\begin{proof}
  Let $F$ be a prime $n$-filter with $x, y \notin F$, and let $G$ and $H$ be the principal filters generated by $x$ and $y$. Then $G \nsubseteq F$ and $H \nsubseteq F$, therefore $G \cap H \nsubseteq F$ as witnessed by some $z \in G \cap H$ such that $z \notin F$. In other words, $x, y \leq z \notin F$, so $x \vee y \notin F$ and the complement of $F$ is an ideal. Conversely, suppose that the complement of $F$ is an ideal and $G$ and $H$ are $n$-filters such that $G \nsubseteq F$ and $H \nsubseteq F$ as witnessed by $x \in G \setminus F$ and $y \in H \setminus F$. Because the complement of $F$ is an ideal, $x \vee y \notin F$. But $x \vee y \in G \cap H$, so $G \cap H \nsubseteq F$.

  To prove the left-to-right direction of the second claim, let $F$ be an $m$-prime filter and $x_{i} \notin F$ for $i \in I$ be a finite family such that $x_{J} \assign \bigvee \set{x_{j}}{j \in J} \notin F$ for each $J \bsubseteq{n} I$. Let $G_{i}$ be the principal filter generated by $x_{i}$ for $i \in I$. Then $x_{J}$ witnesses that $\bigcap_{j \in J} G_{j} \nsubseteq F$ for each $J \bsubseteq{n} I$. Because $F$ is an $m$-prime filter, it follows that $\bigcap_{i \in I} G_{i} \nsubseteq F$, so there is some $z \notin F$ such that $x_{i} \leq z$ for each~$i \in I$. Thus $\bigvee \set{x_{i}}{i \in I} \notin F$ and the complement of $F$ is an $m$-ideal.

  Conversely, consider filters $G_{i}$ for $i \in \range{m+1}$ with $\bigcap_{j \neq i} G_{j} \nsubseteq F$ for $i \in \range{m+1}$, as witnessed by the elements $x_{i} \in \bigcap_{j \neq i} G_{j} \setminus F$. Because $G_{i}$ is a filter, $y_{i} \assign \bigwedge \set{x_{j}}{j \neq i} \in G_{i}$. Then $y_{j} \leq x_{i} \notin F$ for each $j \neq i$. Because the complement of $F$ is an $m$-ideal, $z \assign \bigvee \set{y_{i}}{i \in \range{m+1}} \notin F$. But $z \in \bigcap_{i \in \range{m+1}} G_{i}$, so $\bigcap_{i \in \range{m+1}} G_{i} \nsubseteq F$.
\end{proof}

  The following equivalences thus provide a working definition of prime $n$-filters and $n$-prime filters on lattices.

\begin{fact}\label{fact: n prime ideal complement}
  An $n$-filter $F$ on a lattice is prime if and only if
\begin{align*}
  x \vee y \in F & \implies x \in F \text{ or } y \in F.
\end{align*}
  A filter $F$ on a lattice is $m$-prime if and only if
\begin{align*}
  x_1 \vee \dots \vee x_{m+1} \in F \implies \bigvee_{j \neq i} x_{j} \in F \text{ for some } i \in \range{m+1}.
\end{align*}
\end{fact}

  While the complement of an $m$-prime $n$-filter is always an $m$-ideal, we shall see that the converse implication fails (even in Boolean algebras): a $2$-filter whose complement is a $2$-ideal need not be a $2$-prime $2$-filter.

  The following observation shows how ubiquitous ($m$-prime) $n$-filters are.

\begin{fact} \label{fact: finite upsets are n filters}
  Each upset of a finite semilattice (lattice) is an ($m$-prime) $n$-filter for some $n$ (and $m$). Each element of a finite semilattice is $m$-prime for some $m$.
\end{fact}

\begin{proof}
  If an upset $F \subseteq \alg{S}$ is not an $n$-filter, then this must be witnessed by at least $n+1$ distinct elements of $\alg{S}$. Moreover, the failure of an $n$-filter to be $m$-prime must be witnessed by at least $m+1$ distinct subsets, and similarly the failure of an element of a semilattice $\alg{S}$ to be $m$-prime must be witnessed by at least $m+1$ distinct elements of $\alg{S}$.
\end{proof}

   Let us stress that we count the empty filter as well as the total filter as $m$-prime $n$-filters. This is because we wish to identify prime $n$-filters on distributive lattices (and $m$-prime $n$-filters on finite distributive lattices) as the homomorphic preimages of a certain canonical prime ($m$-prime) $n$-filter.

\subsection{Constructing \texorpdfstring{$n$-filters}{n-filters}}

  Homomorphisms, subalgebras, and products interact with (prime) $n$-filters in the expected ways.

\begin{fact} \label{fact: hom preimages}
  Let $h\colon \alg{S} \rightarrow \alg{T}$ be a semilattice homomorphism and $F$ be an $n$-filter on $\alg{T}$. Then $h^{-1}[F]$ is an $n$-filter on $\alg{S}$.
\end{fact}

\begin{fact}
  Let $h\colon \alg{L} \rightarrow \alg{M}$ be a lattice homomorphism and $F$ be a prime $n$-filter ($m$-prime filter) on $\alg{M}$. Then $h^{-1}[F]$ is a prime $n$-filter ($m$-prime filter) on~$\alg{M}$.
\end{fact}

\begin{fact}
  Let $\alg{S}$ be a subsemilattice of $\alg{T}$ and $F$ be an $n$-filter on $\alg{T}$. Then the restriction of $F$ to $\alg{S}$ is an $n$-filter on $\alg{S}$.
\end{fact}

\begin{fact}
  Let $\alg{L}$ be a sublattice of $\alg{M}$ and $F$ be a prime $n$-filter ($m$-prime filter) on $\alg{M}$. Then the restriction of $F$ to $\alg{L}$ is a prime $n$-filter ($m$-prime filter) on $\alg{L}$.
\end{fact}

\begin{fact}
  Let $\alg{S}_{i}$ for $i \in I$ be a family of semilattices and $F_{i}$ be an $n$-filter on~$\alg{S}_{i}$ for each $i \in I$. Then $F \assign \prod_{i \in I} F_{i}$ is an $n$-filter on $\alg{S} \assign \prod_{i \in I} \alg{S}_{i}$.
\end{fact}

  To show that restricting an $m$-prime $n$-filter to a subsemilattice yields an $m$-prime $n$-filter, we shall need to make a certain assumption about how the smaller semilattice sits inside the full semilattice.

  We say that a subposet $P$ of a poset $Q$ is an \emph{ideal subposet} of~$Q$ if the restriction to $P$ of each ideal on $Q$ is an ideal on $P$. Equivalently, $P$~is an ideal subposet of $Q$ if for each $x, y \in P$ and each $u \in Q$ such that ${x, y \leq u}$ in $Q$ there is some $z \in P$ such that $x, y \leq z \leq u$. In particular, each subsemilattice of a join semilattice is an ideal subposet. A~subsemilattice of a meet semilattice which is an ideal subposet will be called an ideal subsemilattice. For example, $\alg{S}$ is isomorphic to an ideal subsemilattice of the lattice $\Fi \alg{S}$ of all filters on $\alg{S}$ with respect to the embedding assigning to each $a \in \alg{A}$ the principal filter generated by $a$.

\begin{fact} \label{fact: restriction of m prime n filter}
  Let $\alg{S}$ be an ideal subsemilattice of $\alg{T}$ and let $F$ be an $m$-prime $n$-filter on $\alg{T}$. Then the restriction of $F$ to $\alg{S}$ is an $m$-prime $n$-filter on $\alg{S}$.
\end{fact}

\begin{proof}
  We already know that the restriction of $F$ to $\alg{S}$ is an $n$-filter. Let us call it $G$. To prove that $G$ is an $m$-prime $n$-filter, suppose that $\bigcap_{i \in I} G_{i} \subseteq G$ for some non-empty finite family of $n$-filters $G_{i}$ with $i \in I$. Let $F_{i}$ be the upset generated by $G_{i}$ in~$\alg{T}$. Then each $F_{i}$ is an $n$-filter. Now if $a \in \bigcap_{i \in I} F_{i}$, then for each $i \in I$ there is some $g_{i} \in G_{i}$ such that $g_{i} \leq a$. Because $\alg{S}$ is an ideal subposet of $\alg{T}$, there is some $g \in \alg{S}$ such that $g_{i} \leq g$ for each $i \in I$ and $g \leq a$. Thus $a \geq g \in \bigcap_{i \in I} G_{i} \subseteq G \subseteq F$. We have therefore proved that $\bigcap_{i \in I} F_{i} \subseteq F$. Because $F$ is $m$-prime, there is some $J \bsubseteq{m} I$ such that $\bigcap_{j \in J} F_{j} \subseteq F$. Restricting both sides to $\alg{S}$ now yields $\bigcap_{j \in J} G_{j} \subseteq G$, therefore $G$ is an $m$-prime $n$-filter.
\end{proof}

  Conversely, we may wish to extend an $n$-filter on a subsemilattice (or more generally, a subposet which is a semilattice) to the full semilattice.

\begin{fact}
  Let $F$ be an $n$-filter on a semilattice $\alg{S}$ and let $\alg{S}$ be an ideal subposet of a semilattice $\alg{T}$. Then the upset generated by $F$ in $\alg{T}$ is an $n$-filter on $\alg{T}$.
\end{fact}

\begin{proof}
  Let $G$ be the upset generated by $F$ in $\alg{T}$. To show that $G$ is an \mbox{$n$-filter}, suppose that there are $x_{1}, \dots, x_{n+1} \in \alg{T}$ with $\bigwedge_{j \neq i} x_{j} \in G$ for each $i \in \range{n+1}$. That is, there are $y_{1}, \dots, y_{n+1} \in F$ such that $y_{i} \leq \bigwedge_{j \neq i} x_{i}$. We~need to find $z \in F$ such that $z \leq x_{1} \wedge \dots \wedge x_{n+1}$. Because $\alg{S}$ is an ideal subposet of~$\alg{T}$, for each $i \in \range{n+1}$ there is some $z_{i}$ such that $y_{j} \leq z_{i} \leq x_{i}$ for each $j \neq i$. Then $\bigwedge_{j \neq i} z_{j} \geq y_{i} \in F$ for each $i \in \range{n+1}$, therefore $z \assign \bigwedge_{i \in \range{n+1}} z_{i} \in F$, where the meets take place in $\alg{S}$. Moreover, $z \leq z_{i} \leq x_{i}$ for each $i \in \range{n+1}$.
\end{proof}

\begin{fact}
  Let $F$ be an $n$-filter on a sublattice $\alg{L}$ of a lattice $\alg{M}$. Then the upset generated by $F$ in $\alg{M}$ is an $n$-filter on~$\alg{M}$.
\end{fact}

  The simplest way to construct $n$-filters is to take unions of filters.

\begin{fact} \label{fact: union of m prime n filters}
  Let $F_{i}$ for $i \in I$ be a finite family of upsets of $\alg{S}$ such that $F_{i}$ is an $n_{i}$-filter for each $i \in I$. Then $F \assign \bigcup_{i \in I} F_{i}$ is an $n$-filter for $n \assign \sum_{i \in I} n_{i}$.
\end{fact}

\begin{proof}
  We may assume that $I$ is non-empty and $0 < n_{i}$ for each $i \in I$ (in particular, $n_{i} < n$ for each $n_{i}$): if $n_{i} = 0$, then either $F_{i} = \alg{S}$, in which case $F = \alg{S}$ is a $0$-filter, or $F_{i} = \emptyset$, in which case we may remove $F_{i}$ from the family.

  Now consider $x_{1}, \dots, x_{n+1} \in \alg{S}$ such that $y_{i} \assign \bigwedge_{j \neq i} x_{j} \in F$ for $i \in \range{n+1}$. Because $n = \sum_{i \in I} n_{i}$, there is some $k \in \range{n+1}$ such that at least $n_{k}+1 \geq 2$ of the meets $y_{i}$ belong to $F_{k}$, say $y_{j}$ and $y_{k}$ for $j \neq k$. But then the meet of each subset of $\{ x_{1}, \dots, x_{n+1} \}$ of cardinality at most $n_{k} < n$ is a submeet of $y_{j}$ or $y_{k}$, therefore it belongs to $F_{k}$. Because $F_{k}$ is an $n_{k}$-filter, $x_{1} \wedge \dots \wedge x_{n+1} \in F_{k} \subseteq F$.
\end{proof}

\begin{fact} \label{fact: union of n filters}
  The union of a family of $n$ filters is an $n$-filter.
\end{fact}

  We shall see that the converse is true for \emph{prime} $n$-filters: each of them decomposes into a union of $n$ filters. However, there is no bound on how many filters are required to obtain a general $n$-filter: the~upset \mbox{generated} by the $\kappa$~coatoms of the Boolean lattice~$\BA{\kappa} \assign (\BA{1})^{\kappa}$ is a $2$-filter which is not a union of fewer than $\kappa$ filters.

\subsection{Generating \texorpdfstring{$n$-filters}{n-filters}}

  We now wish to explicitly describe the $n$-filter on $\alg{S}$ generated by some $U \subseteq \alg{S}$. We denote this $n$-filter $\fg{n}{U}$. Because $\fg{n}{U}$ always extends the upset generated by~$U$, it suffices to restrict our discussion to cases where $U$ is an upset.

  The $1$-filter generated by an upset $U \subseteq \alg{S}$ admits a simple description:
\begin{align*}
  a \in \fg{1}{U} & \iff \text{there is some non-empty finite $X \subseteq U$ such that $\bigwedge X \leq a$.}
\end{align*}
  To extend this description to $n$-filters, we impose distributivity on $\alg{S}$. Recall that a meet semilattice is \emph{distributive} if
\begin{align*}
  x \wedge y \leq z \implies x' \wedge y' = z \text{ for some $x' \geq x$ and $y' \geq y$}.
\end{align*}
  We say that $Y \subseteq \alg{S}$ \emph{refines} $X \subseteq \alg{S}$ if for all $y \in Y$ there is $x \in X$ such that~$x \leq y$.

\begin{lemma}[Common refinement]
  Let $\alg{S}$ be a distributive meet semilattice and let $Y_{i} \subseteq \alg{S}$ for $i \in \range{n}$ be a family of non-empty finite sets. If $\bigwedge Y_{i} \leq x$ for each~$Y_{i}$, then there is some $Y \subseteq \alg{S}$ such that $\bigwedge Y = x$ and moreover $Y$ refines each $Y_{i}$.
\end{lemma}

\begin{proof}
  This holds by a straightforward induction over $n$.
\end{proof}

\begin{lemma}[Generating $n$-filters] \label{lemma: generating n filters}
  Let $\alg{S}$ be a distributive semilattice and $U$ be an upset of~$\alg{S}$. Then
\begin{align*}
  a \in \fg{n}{U} \iff  & \text{there is some non-empty finite $X \subseteq U$ such that} \\
  & \text{$\bigwedge Y \in U$ for each $Y \bsubseteq{n} X$ and $\bigwedge X \leq a$.}
\end{align*}
  We may equivalently replace $\bigwedge X \leq a$ by $\bigwedge X = a$ in the above.
\end{lemma}

\begin{proof}
  Let $F \subseteq \alg{S}$ be the upset defined by this condition, and let us call a non-empty finite set $X \subseteq \alg{S}$ admissible if $\bigwedge Y \in U$ for each $Y \bsubseteq{n} X$. Because $F$ is distributive and $U$ is an upset, $F$ consists precisely of the meets of admissible subsets of $\alg{S}$. Clearly $F \subseteq \fg{n}{U}$. Conversely, we prove that $F \supseteq U$ is an $n$-filter.

  Consider $x_1, \dots, x_{n+1} \in F$ such that $y_{i} \assign \bigwedge_{j \neq i} x_{j} \in F$ for each $i \in \range{n+1}$. We need to prove that $\bigwedge_{i} x_{i} \in F$. That is, we need to express $\bigwedge_{i} x_{i}$ as the meet of an admissible set. By the definition of $F$, for each $i \in \range{n+1}$ there is an admissible set $Y_{i}$ with $\bigwedge Y_{i} = y_{i}$. In particular, $\bigwedge Y_{i} \leq x_{j}$ for each $j \neq i$.

  Applying the previous lemma to the family $Y_{j}$ for $j \neq i$ and to $x \assign x_{i}$ yields a set $Z_{i}$ which refines each $Y_{j}$ with $j \neq i$ such that $\bigwedge Z_{i} = x_{i}$. It follows that $\bigwedge \bigcup \set{Z_{i}}{i \in \range{n+1}} = \bigwedge_{i \in \range{n+1}} x_{i}$. It thus suffices to show that $\bigcup \set{Z_{i}}{i \in \range{n+1}}$ is an admissible set. If $Z \bsubseteq{n} \bigcup \set{Z_{i}}{i \in \range{n+1}}$, then for each $z \in Z$ there is $i \in \range{n+1}$ such that $z \in Z_{i}$. Thus for each $j \neq i$ there is an element of $Y_{j}$ which lies below $z$. Crucially, since $Z$ has at most $n$ elements but $j$ varies over $\range{n+1}$, we may choose the same $Y_{j}$ for each $z \in Z$. That is, there is some $Y_{j}$ such that for each $z \in Z$ there is $y \in Y_{j}$ with $y \leq z$. But then $\bigwedge Z \in U$ follows from the fact that the meet of each non-empty subset of $Y_{j}$ with at most $n$ elements lies in $U$.
\end{proof}

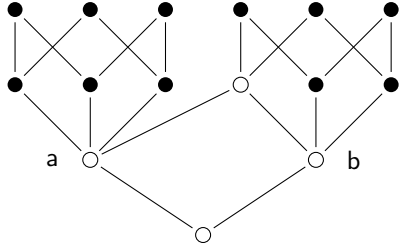
\begin{figure}
\caption{Generating $n$-filters in semilattices}
\label{fig: generating}
\vskip 15pt
\begin{center}
\begin{tikzpicture}[scale=1, dot/.style={inner sep=2.5pt,outer sep=2.5pt}, solid/.style={circle,fill,inner sep=2pt,outer sep=2pt}, empty/.style={circle,draw,inner sep=2pt,outer sep=2pt}]
  \node (0) at (0,0) [empty] {};
  \node (11) at (-1,1) [solid] {};
  \node (12) at (0,1) [solid] {};
  \node (13) at (1,1) [solid] {};
  \node (21) at (-1,2) [solid] {};
  \node (22) at (0,2) [solid] {};
  \node (23) at (1,2) [solid] {};
  \draw[-] (0) edge (11) edge (12) edge (13);
  \draw[-] (21) edge (11) edge (12);
  \draw[-] (22) edge (11) edge (13);
  \draw[-] (23) edge (12) edge (13);
  \node (b0) at (3,0) [empty] {};
  \node (b11) at (2,1) [empty] {};
  \node (b12) at (3,1) [solid] {};
  \node (b13) at (4,1) [solid] {};
  \node (b21) at (2,2) [solid] {};
  \node (b22) at (3,2) [solid] {};
  \node (b23) at (4,2) [solid] {};
  \node (0l) at (-0.5, 0) [dot] {$\mathsf{a}$};
  \node (b0r) at (3.5, 0) [dot] {$\mathsf{b}$};
  \draw[-] (b0) edge (b11) edge (b12) edge (b13);
  \draw[-] (b21) edge (b11) edge (b12);
  \draw[-] (b22) edge (b11) edge (b13);
  \draw[-] (b23) edge (b12) edge (b13);
  \draw[-] (0) edge (b11);
  \node (bot) at (1.5, -1) [empty] {};
  \draw[-] (bot) edge (0) edge (b0);
\end{tikzpicture}
\end{center}
\end{figure}

  The above lemma does not hold in all semilattices. Figure~\ref{fig: generating} shows why. Let~$U$ be the upset consisting of the solid dots. Then there is no $X \subseteq U$ such that $\bigwedge Y \in U$ for each $Y \bsubseteq{2} X$ and $\bigwedge X \leq \mathsf{b}$. There is, however, some such~$X$ with $\bigwedge X = \mathsf{a}$, therefore $\mathsf{a} \in \fg{2}{U}$. This fills the empty dot above $\mathsf{b}$, i.e.\ shows that $\mathsf{a} \vee \mathsf{b} \in \fg{2}{U}$. It immediately follows that $\mathsf{b} \in \fg{2}{U}$. Appending a top element to this semilattice shows that the lemma may even fail in a finite lattice.

  Informally speaking, the $n$-filter generated by $X$ on a general semilattice $\alg{S}$ is formed by applying the construction of the above lemma $\omega$ times.

  Recall that in general the $n$-filters of a meet semilattice $\alg{S}$ form an algebraic lattice~$\Fi_{n} \alg{S}$. If $\alg{S}$ is distributive, we can in fact say more about $\Fi_{n} \alg{S}$.

\begin{fact}
  If $\alg{S}$ is a distributive semilattice, then $\Fi_{n} \alg{S}$ is a distributive lattice.
\end{fact}

\begin{proof}
  Let $F$, $G$, $H$ be $n$-filters on $\alg{S}$ and let $G \sqcup H$ denote the $n$-filter generated by $G \cup H$. Suppose that $x \in F \cap (G \sqcup H)$. Then there is $Y \subseteq \alg{S}$ such that $\bigwedge Y = x$ and $\bigwedge Z \in G \cup H$ for each $Z \bsubseteq{n} Y$. Since $x \in F$, then in fact $\bigwedge Z \in F \cap (G \cup H) = (F \cap G) \cup (F \cap H)$, hence $Y$ witnesses that $x \in (F \cap G) \sqcup (F \cap H)$.
\end{proof}

\section{\texorpdfstring{$n$-filters}{n-filters} as intersections of homomorphic preimages}

  We now show that each $n$-filter is an intersection of homomorphic preimages of the $n$-filter $\nonempty{n} \subseteq \BA{n}$.

\subsection{Prime \texorpdfstring{$n$-filters}{n-filters} as homomorphic preimages}

  We first exhibit each prime $n$-filter as a homomorphic preimage of $\nonempty{n} \subseteq \BA{n}$. In order to better understand how the $n$-filter $\nonempty{n}$ arises in this context, we recall the notion of the dual product of a family of structures. Here by a \emph{structure} we simply mean a pair $\pair{\alg{A}}{F}$ consisting of an algebra $\alg{A}$ and a set $F \subseteq \alg{A}$. Given a family of structures $\pair{\alg{A}_{i}}{F_{i}}$ with $i \in I$ in some common algebraic signature, their \emph{(direct) product} is defined as
\begin{align*}
  \prod_{i \in I} \pair{\alg{A}_{i}}{F_{i}} \assign \Big \langle {\prod_{i \in I} \alg{A}_{i}}, {\bigcap_{i \in I} \pi_{i}^{-1}[F_{i}]} \Big \rangle,
\end{align*}
  where $\pi_{i}$ are the projections. The \emph{dual (direct) product} of the family $\pair{\alg{A}_{i}}{F_{i}}$, studied by Badia and Marcos~\cite{badia+marcos18}, is less well known. It is defined as the structure
\begin{align*}
  \bigotimes_{i \in I} \pair{\alg{A}_{i}}{F_{i}} \assign \Big \langle {\prod_{i \in I} \alg{A}_{i}}, {\bigcup_{i \in I} \pi_{i}^{-1}[F_{i}]} \Big \rangle.
\end{align*}
  The two constructions are related by De~Morgan duality:
\begin{align*}
  \overline{\pair{\alg{A}}{F} \otimes \pair{\alg{B}}{G}} = \overline{\pair{\alg{A}}{F}} \times \overline{\pair{\alg{B}}{G}}, \text{ where } \overline{\pair{\alg{A}}{F}} \assign \pair{\alg{A}}{\alg{A} \setminus F}.
\end{align*}
  The construction $\pair{\alg{A}}{F} \mapsto \pair{\alg{A}}{\alg{A} \setminus F}$ in our case corresponds to switching from an upset (prime $n$-filter) to a downset ($n$-prime ideal) and vice versa.

  In order to describe prime $n$-filters and $n$-prime filters on distributive lattices, we will need to construct homomorphisms into direct products and dual direct products of structures. The following simple observation will be crucial.

  Let us recall that a \emph{strict} homomorphism of structures $h\colon \pair{\alg{A}}{F} \to \pair{\alg{B}}{G}$ is a homomorphism of algebras $h\colon \alg{A} \to \alg{B}$ such that $F = h^{-1}[G]$.

\begin{fact} \label{fact: hom into dual product}
  Let $h_{i}\colon \pair{\alg{A}}{F_{i}} \to \pair{\alg{B}_{i}}{G_{i}}$ be a family of strict homomorphisms for $i \in I$ and let $h\colon \alg{A} \to \prod_{i \in I} \alg{B}_{i}$ be the product of these homomorphisms. Then the following are strict homomorphisms:
\begin{align*}
  & h\colon \pair{\alg{A}}{\bigcap_{i \in I} F_{i}} \to \prod_{i \in I} \pair{\alg{B}_{i}}{G_{i}}, & & h\colon \pair{\alg{A}}{\bigcup_{i \in I} F_{i}} \to \bigotimes_{i \in I} \pair{\alg{B}_{i}}{G_{i}}.
\end{align*}
\end{fact}

  If we now apply these two constructions to the structure $\BAm{1} \assign \pair{\BA{1}}{\{ \True \}}$, where $\BA{1}$ is the two-element distributive lattice, we obtain the structures
\begin{align*}
 \pair{\BA{n}}{\{ \True \}} & = \prod_{i \in \range{n}} \pair{\BA{1}}{\{ \True \}}, & \pair{\BA{n}}{\nonempty{n}} & = \bigotimes_{i \in \range{n}} \pair{\BA{1}}{\{ \True \}}.
\end{align*}
  Recall that $\BA{n} \assign \BA{1}^{n}$ and $\nonempty{n}$ denotes the upset of non-zero elements of $\BA{n}$, with $\BA{0}$ being the trivial lattice and $\nonempty{0} = \emptyset$. We shall use the notation
\begin{align*}
  (\BAm{1})^{n} & \assign \pair{\BA{n}}{\{ \True \}}, & \BAm{n} & \assign \pair{\BA{n}}{\nonempty{n}}.
\end{align*}
  Occasionally, we also use the notation $\BAm{1}^{\Pi n}$ and $\BAm{1}^{\otimes n}$ for direct and dual powers.

\begin{fact} \label{fact: nonempty n is n filter}
  $\nonempty{n}$ is a prime $n$-filter on $\BA{n}$. It is not an $m$-filter for any $m < n$.
\end{fact}

\begin{proof}
  It is a prime $n$-filter because it is a union of the $n$ prime filters generated by the atoms of $\BA{n}$. It is not an $m$-filter for $m < n$ because the meet of each set of $m$ coatoms lies in $\nonempty{1}$ but the meet of all coatoms does not.
\end{proof}

  We say that an upset $F$ on a lattice $\alg{L}$ is a \emph{homo\-morphic preimage} of an upset $G$ on a lattice $\alg{M}$ if there is a lattice homo\-morphism $h\colon \alg{L} \to \alg{M}$ such that $F = h^{-1}[G]$. An upset $F$ is \emph{$n$-generated} if it is the upward closure of a set of cardinality at most $n$. The following lemma holds equally well for semilattices.

\begin{lemma} \label{lemma: union of n filters}
  The following are equivalent for each upset $U$ of a lattice $\alg{L}$:
\begin{enumerate}[(i)]
  \item for all $a_{1}, \dots, a_{n+1} \in U$ there are $i \neq j$ such that $a_{i} \wedge a_{j} \in U$,
  \item $U$ is a union of at most $n$ filters on $\alg{L}$,
  \item $U$ is a homomorphic preimage of an $n$-generated upset.
\end{enumerate}
\end{lemma}

\begin{proof}
  (i) $\Rightarrow$ (ii): let $n$ be the least number which satisfies (i). If~${n = 1}$, then $U$ is a filter, so (ii) holds. Otherwise, there are $b_{1}, \dots, b_{n} \in U$ such that ${b_{i} \wedge b _{j} \notin U}$ whenever $i \neq j$. We can now extend the principal filters generated by $b_{1}, \dots, b_{n}$ to maximal filters $F_{1}, \dots, F_{n} \subseteq U$. That is, if $x \notin F_{i}$, then $x \wedge f_{i} \notin U$ for some $f_{i} \in F_{i}$. It suffices to show that $U = F_{1} \cup \dots \cup F_{n}$. Suppose therefore that $x \notin F_{1} \cup \dots \cup F_{n}$. Then there are $c_{i} \in F_{i}$ such that $x \wedge c_{i} \notin U$. We~may assume without loss of generality that $c_{i} \leq b_{i}$, taking $b_{i} \wedge c_{i}$ instead of $c_{i}$ if necessary. But then $c_{i} \wedge c_{j} \notin F$ whenever $i \neq j$. The set $\{ x, c_{1}, \dots, c_{n} \}$ therefore contradicts the assumption (i).

  (ii) $\Rightarrow$ (iii) $\Rightarrow$ (iv): let $U = F_{1} \cup \dots \cup F_{k}$, where $F_{i}$ are filters on $\alg{L}$. Consider the embedding $\iota\colon \alg{L} \into \Fi_{1} (\alg{L})$ of $\alg{L}$ into the lattice $\Fi_{1} (\alg{L})$ of filters on~$\alg{L}$ ordered by reverse inclusion where $\iota(a)$ is the principal filter generated by $a$. This map is a lattice homomorphism. Then $a \in F_{i}$ if and only if $\iota(a) \subseteq F_{i}$, so $a \in F$ if and only if $\iota(a) \subseteq F_{i}$ for some $i \in \range{k}$. The~upset~$U$ is thus the preimage of the upset generated in $\Fi_{1} (\alg{L})$ by the $k$ elements~$F_{i}$.

  (iii) $\Rightarrow$ (i): it suffices to show that each upset $U$ generated by some elements $b_{1}, \dots, b_{n}$ satisfies (i). But for each set $\{ a_{1}, \dots, a_{n+1} \}$ there is some $b_{k}$ such that $b_{k} \leq a_{i}$ and $b_{k} \leq a_{j}$ for some $i \neq j$.
\end{proof}

\begin{lemma}[Homomorphism lemma for $\nonempty{1}$]
  Each prime \mbox{filter} on a distributive lattice~$\alg{L}$ is a homomorphic preimage of the prime filter $\nonempty{1} \subseteq \BA{1}$.
\end{lemma}

  The homomorphism lemma for $\nonempty{1}$ immediately extends to~$\nonempty{n}$. Each family of prime filters $F_{i}$ on $\alg{L}$ for $i \in I$ corresponds to a family of strict homo\-morphisms $h_{i}\colon \pair{\alg{L}}{F_{i}} \to \pair{\BA{1}}{\{ \True \}}$, hence the direct product of this family of homo\-morphisms $h\colon \alg{L} \to \prod_{i \in I} \BA{1}$ yields a strict homomorphism $h\colon \pair{\alg{L}}{\bigcup_{i \in I} F_{i}} \to \pair{\BA{n}}{\nonempty{n}}$ by Fact~\ref{fact: hom into dual product}.

\begin{lemma}[Homomorphism lemma for $\nonempty{n}$] \label{lemma: hom into ba n}
  Let $F_{1}, \dots, F_{n}$ be prime \mbox{filters} on a distributive lattice~$\alg{L}$. Then $F_{1} \cup \dots \cup F_{n}$ is a homomorphic preimage of the prime $n$-filter $\nonempty{n} \subseteq \BA{n}$.
\end{lemma}

\begin{theorem}[Prime $n$-filters on distributive lattices] \label{thm: union of n prime filters}
  The~following conditions are equivalent for each upset~$F$ of a distributive lattice:
\begin{enumerate}[(i)]
  \item $F$ is a prime $n$-filter,
  \item $F$ is a prime upset which is a union of at most $n$ filters,
  \item $F$ is a union of at most $n$ prime filters,
  \item $F$ is a homomorphic preimage of the prime $n$-filter $\nonempty{n} \subseteq \BA{n}$.
\end{enumerate}
\end{theorem}

\begin{proof}
  (i) $\Rightarrow$ (ii): suppose that $\{ a_{1}, \dots, a_{n+1} \} \subseteq F$. By Lemma~\ref{lemma: union of n filters} it suffices to show that $a_{i} \wedge a_{j} \in F$ for some $i \neq j$. Let $b_{i} \assign \bigvee_{j \neq i} a_{j}$. Then for each $Y \bsubseteq{n} \{ b_{1}, \ \dots, b_{n+1} \}$ we have $a_{j} \leq \bigwedge Y$ for some~$a_{j}$. But $a_{j} \in F$ and $F$ is an $n$-filter, thus $b_{1} \wedge \dots \wedge b_{n+1} \in F$. Moreover, $b_{1} \wedge \dots \wedge b_{n+1} = \bigvee_{j} \bigvee_{k \neq j} (a_{j} \wedge a_{k})$, therefore $a_{j} \wedge a_{k} \in F$ for some $j \neq k$ because $F$ is prime.

  (ii) $\Rightarrow$ (iii): suppose that $F$ is prime and $F = F_{1} \cup \dots \cup F_{k}$ for $k \leq n$, where $F_{i}$ are filters. We may assume that this union is irredundant, i.e.\ that for each $i$ there is some $c_{i} \in F_{i}$ such that $c_{i} \notin F_{j}$ for $j \neq i$. Let $a \vee b \in F_{i}$. Then $(a \wedge c_{i}) \vee (b \wedge c_{i}) = (a \vee b) \wedge c_{i} \in F_{i}$, so either $a \wedge c_{i} \in F$ or $b \wedge c_{i} \in F$. But $a \wedge c_{i} \notin F_{j}$ for $j \neq i$, thus $a \wedge c_{i} \in F_{i}$ or $b \wedge c_{i} \in F_{i}$. Therefore $a \in F_{i}$ or $b \in F_{i}$.

  (iii) $\Rightarrow$ (iv) $\Rightarrow$ (ii) $\Rightarrow$ (i): Lemma~\ref{lemma: hom into ba n} and Facts~\ref{fact: hom preimages} and~\ref{fact: union of n filters}.
\end{proof}

  An analogous fact holds for distributive semilattices.

\begin{theorem}[Prime $n$-filters on distributive semilattices]
  The~following conditions are equivalent for each upset~$F$ of a distributive meet semilattice:
\begin{enumerate}[(i)]
  \item $F$ is a prime $n$-filter,
  \item $F$ is a prime upset which is a union of at most $n$ filters,
  \item $F$ is a prime homomorphic preimage of the $n$-filter $\nonempty{n} \subseteq \BA{n}$.
\end{enumerate}
\end{theorem}

\begin{proof}
  The equivalence (ii) $\Leftrightarrow$ (iii) and the implication (ii) $\Rightarrow$ (i) are entirely analogous to the lattice case, taking into account that semilattice filters are precisely the homomorphic preimages of $\{ \True \} \subseteq \BA{1}$. It remains to show that each prime $n$-filter $F$ on a semilattice $\alg{S}$ is a union of at most $n$ filters. To this end, let $\alg{L}$ be the order dual of the lattice $\Fi \alg{S}$ of filters on $\alg{S}$. Since $\alg{S}$ is distributive, we know that $\Fi \alg{S}$, hence also $\alg{L}$, is a distributive lattice. The map $\iota\colon \alg{S} \to \alg{L}$ which assigns to each $a \in \alg{S}$ the principal upset generated by $a$ then embeds $\alg{S}$ into $\alg{L}$ (as a meet semilattice). Let $\overline{F} \assign \set{G \in \Fi \alg{S}}{G \subseteq F}$. It is straight\-forward to show that $\overline{F}$ is a prime $n$-filter on $\alg{L}$ such that $\iota(a) \in \overline{F}$ if and only if $a \in F$. Applying the characterization of prime $n$-filters on distributive lattices, $\overline{F}$ is a union of $n$ filters on $\alg{L}$, therefore $F = \iota^{-1}[\overline{F}]$ is a union of $n$ filters on $\alg{S}$.
\end{proof}

  Distributivity is essential in the above theorems. The upset $P(\Mfive)$ of non-zero elements is a prime $2$-filter on the five-element diamond $\Mfive$ but it is not a union of at most $2$ filters, and accordingly it is not a homomorphic preimage of $\nonempty{2} \subseteq \BA{2}$. On the other hand, the upset generated by the coatoms of the five-element pentagon $\Nfive$ is a prime union of $2$ filters, but it is not a union of prime filters, and accordingly it is not a homomorphic preimage of $\nonempty{2} \subseteq \BA{2}$.

  Dualizing the above characterization of prime $n$-filters on distributive lattices yields an analogous characterization of prime $n$-ideals. But prime $n$-ideals are precisely the complements of $n$-prime filters, therefore the theorem can equivalently be phrased in terms of $n$-prime filters.

\begin{theorem}[$n$-prime filters on distributive lattices]
  The following conditions are equivalent for each upset $F$ of a distributive lattice:
\begin{enumerate}[(i)]
  \item $F$ is an $n$-prime filter,
  \item $F$ is a filter and $F$ is an intersection of at most $n$ prime upsets,
  \item $F$ is an intersection of at most $n$ prime filters,
  \item $F$ is a homomorphic preimage of the $n$-prime filter $\{ \True \} \subseteq \BA{n}$.
\end{enumerate}
\end{theorem}

\subsection{Prime \texorpdfstring{$n$-filter}{n-filter} separation}

  We now exhibit each $n$-filter as an intersection of prime $n$-filters.

\begin{theorem}[Prime $n$-filter separation] \label{thm: prime n filter separation}
  Let $F$ be an $n$-filter on a distributive lattice which is disjoint from an ideal $I$. Then $F$ extends to a prime $n$-filter which is disjoint from $I$.
\end{theorem}

\begin{proof}
  If $I = \emptyset$, the claim is true. Otherwise, let $\theta_{I}$ be the congruence which collapses all elements of $I$. That is, $\pair{a}{b} \in  \theta_{I}$ if and only if $a \vee i = b \vee i$ for some $i \in I$. If $\pi_{I}\colon \alg{L} \to \alg{L} / \theta_{I}$ is the projection map and $G \subseteq \alg{L} / \theta_{I}$ is a proper prime $n$-filter extending $\pi_{I}[F]$, then $\pi^{-1}[G]$ is a prime $n$-filter separating $F$ from $I$. It~thus suffices to prove that each proper $n$-filter on a distributive lattice with a lower bound extends to a proper prime $n$-filter.

  Consider a lower bounded distributive lattice $\alg{L}$ with a proper $n$-filter $F$ on~$\alg{L}$. Then~$\pair{\alg{L}}{F}$ embeds into an ultraproduct $\prod_{U} \pair{\alg{L}_{i}}{F_{i}}$ of its finitely generated (hence finite) substructures $\pair{\alg{L}_{i}}{F_{i}}$ for $i \in I$, where $F_{i}$ is the restriction of $F$ to~$\alg{L}_{i}$. Suppose that for each $F_{i}$ there is a proper prime $n$-ideal $G_{i} \supseteq F_{i}$ on $\alg{L}_{i}$. Being a prime \mbox{$n$-filter} can be expressed by a first-order sentence, therefore the ultraproduct $\prod_{U} \pair{\alg{L}_{i}}{G_{i}}$ satisfies this sentence. Restricting the prime $n$-filter on this ultra\-product to~$\alg{L}$ (as canonically embedded in~$\prod_{U} \alg{L}_{i}$) now yields the desired prime $n$-filter on $\alg{L}$ extending $F$. It therefore suffices to prove that each proper $n$-filter on a finite distributive lattice extends to a proper prime $n$-filter.

  Each finite distributive lattice $\alg{L}$ embeds into a finite Boolean \mbox{algebra}~$\BA{m}$ for some $m$. The upset generated by a proper $n$-filter $F$ on $\alg{L}$ is a proper $n$-filter $G$ on $\BA{m}$. If $G$ extends to a proper prime $n$-filter $G'$, then restricting $G'$ to $\alg{L}$ yields a proper  prime $n$-filter $F' \supseteq F$ on $\alg{L}$. It therefore suffices to prove that each proper $n$-filter on $\BA{m}$ extends to a proper prime $n$-filter.

  Prime~filters on $\BA{m}$ are principal filters generated by atoms, thus by Theorem~\ref{thm: union of n prime filters} prime $n$-filters are precisely the upsets generated by sets of at most $n$~atoms. An upset $F$ on $\BA{m}$ thus extends to a prime $n$-filter if and only if there are atoms $\atom{1}, \dots, \atom{n}$ such that $b \in F$ implies $\atom{i} \leq b$ for some $\atom{i}$. Equivalently, $F$ extends to a prime $n$-filter if and only if there are coatoms $\coatom{1}, \dots, \coatom{n}$ such that $b \in F$ implies $b \nleq \coatom{i}$ for some $\coatom{i}$. Contraposition yields that $F$ does not extend to any prime $n$-filter if and only if for each $n$-tuple of coatoms $\coatom{1}, \dots, \coatom{n}$ there is some $b \in F$ such that $b \leq \coatom{1} \wedge \dots \wedge \coatom{n}$. If $F$ is an $n$-filter, this implies that the meet of all coatoms lies in $F$. An $n$-filter $F$ on $\BA{m}$ which does not extend to any prime $n$-filter is therefore not a proper $n$-filter.
\end{proof}

  Let us also provide a more direct proof of this theorem.

\begin{proof}
  Let $G$ be a maximal $n$-filter disjoint from $I$ which extends $F$. If $x, y \notin G$, then the $n$-filter generated by $F \cup \{ x \}$ intersects $I$ and the $n$-filter generated by $F \cup \{ y \}$ intersects $I$. That is, there is some $i \in I$ such that $i \in \fg{n}{F, x}$ and some $j \in I$ such that $j \in \fg{n}{F, y}$. By the following lemma, $i \vee j \in \fg{n}{F, x \vee y}$, so $x \vee y \notin G$.
\end{proof}

\begin{lemma} \label{lemma: intersect fg n}
  $\fg{n}{F, x} \cap \fg{n}{F, y} = \fg{n}{F, x \vee y}$ in every distributive lattice.
\end{lemma}

\begin{proof}
  Clearly $\fg{n}{F, x \vee y} \subseteq \fg{n}{F, x} \cap \fg{n}{F, y}$. Conversely, let $i \in \fg{n}{F, x} \cap \fg{n}{F, y}$. There are non-empty finite $P, Q \subseteq \alg{L}$ such that $\bigwedge P \wedge x \leq i$ and $\bigwedge Q \wedge y \leq i$, where (i) $R \bsubseteq{n} P$ and $S \bsubseteq{n} Q$ implies $\bigwedge R \in F$ and $\bigwedge S \in F$, and (ii) $T \bsubseteq{n-1} P$ and $U \bsubseteq{n-1} Q$ implies $\bigwedge T \wedge x \in F$ and $\bigwedge U \wedge y \in F$. Take $z \assign (\bigwedge P \wedge x) \vee (\bigwedge Q \wedge y)$. Then $z \leq i$. We show that $z \in \fg{n}{F, x \vee y}$.

  The~element $z$ is the meet of a set $Z$ consisting of elements of four types: $p \vee q$ for $p \in P$ and $q \in Q$, $p \vee y$ for $p \in P$, $x \vee q$ for $q \in Q$, and $x \vee y$. It suffices to prove that the meet of each subset of $Z$ of cardinality at most $n$ lies in $F$ or above $x \vee y$. This follows by a straightforward case analysis using the assumptions that $\bigwedge R, \bigwedge S, \bigwedge T \wedge x, \bigwedge U \wedge y \in F$ for each $R \bsubseteq{n} P$ and each $S \bsubseteq{n} Q$ and each $T \bsubseteq{n-1} P$ and $U \bsubseteq{n-1} Q$. For~example, if elements of all four types are represented in a subset of $Z$ of cardinality at most $n$, we use the bounds $p \leq p \vee q$, $p \leq p \vee y$, $x \leq x \vee q$, and $x \leq x \vee y$ (or the bounds $q \leq p \vee q$, $y \leq p \vee y$, $q \leq x \vee q$, and $y \leq x \vee y$) to show that the meet of this subset lies in $F$ or above $x \vee y$.
\end{proof}

\begin{corollary}[$n$-filters on distributive lattices]
  The $n$-filters on a distributive lattice are precisely the inter\-sections of prime $n$-filters.
\end{corollary}

\begin{figure}
\caption{The structures $(\BAm{1}^{\otimes 2})^{\Pi 2}$ and $(\BAm{1}^{\Pi 2})^{\otimes 2}$}
\label{fig: dBA22}
\vskip 15pt
\begin{center}
\begin{tikzpicture}[scale=1, dot/.style={inner sep=2.5pt,outer sep=2.5pt}, solid/.style={circle,fill,inner sep=2pt,outer sep=2pt}, empty/.style={circle,draw,inner sep=2pt,outer sep=2pt}]
  \node (0) at (0,0) [empty] {};
  \node (1) at (-1.5,1) [empty] {};
  \node (2) at (-0.5,1) [empty] {};
  \node (3) at (0.5,1) [empty] {};
  \node (4) at (1.5,1) [empty] {};
  \node (12) at (-2.5,2) [empty] {};
  \node (13) at (-1.5,2) [solid] {};
  \node (14) at (-0.5,2) [solid] {};
  \node (23) at (0.5,2) [solid] {};
  \node (24) at (1.5,2) [solid] {};
  \node (34) at (2.5,2) [empty] {};
  \node (123) at (-1.5,3) [solid] {};
  \node (124) at (-0.5,3) [solid] {};
  \node (134) at (0.5,3) [solid] {};
  \node (234) at (1.5,3) [solid] {};
  \node (1234) at (0,4) [solid] {};
  \draw (0) edge (1) edge (2) edge (3) edge (4);
  \draw (1234) edge (123) edge (124) edge (134) edge (234);
  \draw (12) edge (1) edge (2) edge (123) edge (124);
  \draw (13) edge (1) edge (3) edge (123) edge (134);
  \draw (14) edge (1) edge (4) edge (124) edge (134);
  \draw (23) edge (2) edge (3) edge (123) edge (234);
  \draw (24) edge (2) edge (4) edge (124) edge (234);
  \draw (34) edge (3) edge (4) edge (134) edge (234);
\end{tikzpicture}
\qquad
\begin{tikzpicture}[scale=1, dot/.style={inner sep=2.5pt,outer sep=2.5pt}, solid/.style={circle,fill,inner sep=2pt,outer sep=2pt}, empty/.style={circle,draw,inner sep=2pt,outer sep=2pt}]
  \node (0) at (0,0) [empty] {};
  \node (1) at (-1.5,1) [empty] {};
  \node (2) at (-0.5,1) [empty] {};
  \node (3) at (0.5,1) [empty] {};
  \node (4) at (1.5,1) [empty] {};
  \node (12) at (-2.5,2) [solid] {};
  \node (13) at (-1.5,2) [empty] {};
  \node (14) at (-0.5,2) [empty] {};
  \node (23) at (0.5,2) [empty] {};
  \node (24) at (1.5,2) [empty] {};
  \node (34) at (2.5,2) [solid] {};
  \node (123) at (-1.5,3) [solid] {};
  \node (124) at (-0.5,3) [solid] {};
  \node (134) at (0.5,3) [solid] {};
  \node (234) at (1.5,3) [solid] {};
  \node (1234) at (0,4) [solid] {};
  \draw (0) edge (1) edge (2) edge (3) edge (4);
  \draw (1234) edge (123) edge (124) edge (134) edge (234);
  \draw (12) edge (1) edge (2) edge (123) edge (124);
  \draw (13) edge (1) edge (3) edge (123) edge (134);
  \draw (14) edge (1) edge (4) edge (124) edge (134);
  \draw (23) edge (2) edge (3) edge (123) edge (234);
  \draw (24) edge (2) edge (4) edge (124) edge (234);
  \draw (34) edge (3) edge (4) edge (134) edge (234);
\end{tikzpicture}
\end{center}
\end{figure}

  As a consequence of prime $n$-filter separation, we may exhibit each $m$-prime $n$-filter $F$ on a \emph{finite} distributive lattice $\alg{L}$ as the homomorphic preimages of a $m$-prime $n$-filter $\dnonempty{n}{m} \subseteq \dBA{n}{m}$, where $\pair{\dBA{n}{m}}{\dnonempty{n}{m}}$ is the structure defined as
\begin{align*}
  \pair{\dBA{n}{m}}{\dnonempty{n}{m}} & \assign (\BAm{n})^{m} = \prod_{i \in \range{m}} \bigotimes_{j \in \range{n}} \pair{\BA{1}}{\{ \True \}}.
\end{align*}

  These structures are the $m$-th direct powers of the $n$-th dual powers of $\BAm{1}$. The direct square of the dual square of $\BAm{1}$ is shown in Figure~\ref{fig: dBA22} on the left. The structure shown on the right of Figure~\ref{fig: dBA22} is, by contrast, the dual square of the direct square of $\BAm{1}$. This structure provides an example of a $2$-filter whose complement is a $2$-ideal but which is not a $2$-prime $2$-filter. The reader may verify that it cannot be expressed as an intersection of two prime $2$-filters.

\begin{theorem} \label{thm: m prime n filters}
  The $m$-prime $n$-filters on a \emph{finite} distributive lattice are precisely the homomorphic preimages of the $m$-prime $n$-filter $\dnonempty{n}{m} \subseteq \dBA{n}{m}$.
\end{theorem}

\begin{proof}
  Let $F$ be an $m$-prime $n$-filter on a finite distributive lattice $\alg{L}$. By prime $n$-filter separation, $F$ is the intersection of a finite family of prime $n$-filters. Because $F$ is an $m$-prime $n$-filter, we may assume that this family consists of $m$ prime $n$-filters (not necessarily distinct): $F = F_{1} \cap \dots \cap F_{m}$. Each of these prime $n$-filters is in turn a union of at most $n$ prime filters (not necessarily distinct): $F_{i} = G_{i,1} \cup \dots \cup G_{i,n}$. Finally, each prime filter $G_{i, j}$ on $\alg{L}$ yields a strict homomorphism $\pair{\alg{L}}{G_{i, j}} \to \pair{\BA{1}}{\{ \True \}}$. Putting all of this together, Fact~\ref{fact: hom into dual product} now yields a strict homomorphism from $\pair{\alg{L}}{F}$ into $\pair{\dBA{n}{m}}{\dnonempty{n}{m}}$.

  Conversely, let $\pi_{i}\colon \dBA{n}{m} \to \BA{n}$ for $i \in \range{m}$ be the projection maps for $\dBA{n}{m} = (\BA{n})^{m}$. Because $\nonempty{n} \subseteq \BA{n}$ is a prime $n$-filter, so is $\pi_{i}^{-1}[\nonempty{n}]$, hence $\dnonempty{n}{m}$ is an $n$-filter by virtue of being a product of $n$-filters. Moreover, each homomorphic preimage of $\dnonempty{n}{m}$ is an $m$-prime $n$-filter by virtue of being the intersection of the $m$ homomorphic preimages of the prime $n$-filters $\pi_{i}^{-1}[\nonempty{n}]$.
\end{proof}

  In the results proved so far, we may replace the two-element distributive lattice $\BA{1}$ by one of its expansions. By the analysis of Post~\cite{post41}, there are \mbox{exactly} six proper expansions of~$\BA{1}$ up to term equivalence, namely expansions by a constant $\False$ for the bottom element, by a constant $\True$ for the top element, by both constants $\False$ and $\True$, by co-implication ($x \corightarrow y \assign x \wedge \neg y$), by implication ($x \rightarrow y \assign \neg x \vee y$), and by negation ($\neg x$). The~(quasi)varieties generated by these six expansions are the varieties of lower bounded distributive lattices, upper bounded distributive lattices, bounded distributive lattices, generalized Boolean algebras, dual generalized Boolean algebras, and Boolean algebras.

  The results proved so far extend immediately to these six cases, provided that appropriate modifications are made to the definition of $n$-filters. In~the case of upper bounded distributive lattices and dual generalized Boolean \mbox{algebras}, we require that $n$-filters be non-empty. In the case of lower bounded \mbox{distributive} lattices and generalized Boolean algebras, we require that $m$-prime filters be proper. Finally, we impose both of these requirements in the case of bounded distributive lattices and Boolean \mbox{algebras}.

\section{Filter classes}
\label{sec: filter classes}

  Are there other families of upsets of distributive lattices besides $n$-filters which share the basic properties of filters, i.e.\ closure under intersections, homomorphic preimages, and directed unions? To investigate this question, we first introduce some terminology and review some relevant results.

  Let $\class{L}$ be a quasivariety of algebras, i.e.\ a class of algebras in a given signature closed under subalgebras, products, and ultraproducts. Structures of the form $\pair{\alg{A}}{F}$ for $\alg{A} \in \class{L}$ will be called \emph{$\class{L}$-structures}. In the following, $\class{K}$ will denote a class of $\class{L}$-structures. A set $F \subseteq \alg{A} \in \class{L}$ is called a \emph{$\class{K}$-filter} if $\pair{\alg{A}}{F} \in \class{K}$.

  We say that $\class{K}$ forms a \emph{filter class} if it is closed under substructures, products, and strict homomorphic preimages. Equivalently, $\class{K}$ is a filter class if $\class{K}$-filters are closed under homo\-morphic preimages (not necessarily surjective) and arbitrary intersections. If the class~$\class{K}$ is moreover closed under strict homomorphic images, we say that $\class{K}$ is a~\emph{logical class}. A filter class will be called \emph{trivial} if it only contains structures $\pair{\alg{A}}{F}$ where $F$ is the total $\class{K}$-filter, i.e.\ $F = \alg{A}$.

  (Recall that a structure $\pair{\alg{A}}{F}$ is called a strict homomorphic pre\-image of $\pair{\alg{B}}{G}$ and conversely $\pair{\alg{B}}{G}$ is called a strict homomorphic image of $\pair{\alg{A}}{F}$ if there is a strict surjective homomorphism $h\colon \pair{\alg{A}}{F} \to \pair{\alg{B}}{G}$, i.e.\ a surjective homo\-morphism of algebras $h\colon \alg{A} \to \alg{B}$ such that $F = h^{-1}[G]$.)

  A filter class $\class{K}$ of $\class{L}$-structures is said to be \emph{finitary} if $\class{K}$ is closed under ultra\-products, or equivalently if the $\class{K}$-filters on each $\alg{A} \in \class{L}$ are closed under directed unions. We may also phrase this definition in terms of $\class{K}$-filter generation. The $\class{K}$-filter $\Fg_{\class{K}}^{\alg{A}} X$ \emph{generated} by a set $X \subseteq \alg{A} \in \class{L}$ is, of course, the smallest $\class{K}$-filter on $\alg{A}$ which contains $X$. The class $\class{K}$ is then finitary if and only if the closure operator $\Fg_{\class{K}}^{\alg{A}}$ is finitary for each $\alg{A} \in \class{L}$, i.e.\ $\Fg_{\class{K}}^{\alg{A}} X = \bigcup \set{\Fg_{\class{K}}^{\alg{A}} Y}{Y \subseteq \alg{A} \text{ and $Y$ is finite}}$.

  The permutation properties of the class operators in question immediately yield the following theorems, where $\HSinvop(\class{K})$, $\HSop(\class{K})$, $\Sop(\class{K})$, $\Pop(\class{K})$, and $\PUop(\class{K})$ denote the classes of $\class{L}$-structures which are strict homo\-morphic pre\-images, strict homo\-morphic images, substructures, products, and ultraproducts of structures in~$\class{K}$.

\begin{theorem}
  Let $\class{K}$ be a class of $\class{L}$-structures. The filter class generated by $\class{K}$ is $\HSinvop \Sop \Pop(\class{K})$. The logical class generated by $\class{K}$ is $\HSinvop \HSop \Sop \Pop(\class{K})$.
\end{theorem}

\begin{theorem}
 The finitary filter class (finitary logical class) generated by $\class{K}$ is the filter class (logical class) generated by $\PUop(\class{K})$.
\end{theorem}

\begin{corollary}
  The finitary filter class (finitary logical class) generated by $\class{K}$ coincides with the filter class (logical class) generated by $\class{K}$.
\end{corollary}

  Filter classes of $\class{L}$-structures admit an equivalent syntactic characterization as classes \mbox{axiomatized} by implications of a certain form. Consider a proper class of variables. An \emph{atomic formula} has either the form $t \equals u$ or $\Fsymbol(t)$, where $t$ and $u$ are terms over the given class of variables in the algebraic signature of $\class{L}$. The equality symbol is of course interpreted by the equality relation, while $\Fsymbol$ is interpreted in the structure $\pair{\alg{A}}{F}$ by the set $F \subseteq \alg{A}$. By an \emph{implication} (more explicitly, a possibly infinitary \emph{Horn formula}) we shall mean an infinitary formula of the form
\begin{align*}
  \alpha_{i} \text{ for each } i \in I & \implies \beta,
\end{align*}
  where $\beta$ is an atomic formula and $\alpha_{i}$ for $i \in I$ is a set of atomic formulas. The implication is said to be \emph{finitary} if $I$ is finite. In a~\emph{filter implication}, we require that $\beta$ have the form $\Fsymbol(u)$. In an \emph{equality-free} implication, we require that all of the formulas $\alpha_{i}$ and $\beta$ have this form. For example, $F$ is an upset of a distributive latice $\alg{L}$ if and only if $\pair{\alg{L}}{F}$ satisfies the implication
\begin{align*}
  \Fsymbol(x) ~ \& ~ x \wedge y \equals x & \implies \Fsymbol(y).
\end{align*}
  This is a filter implication, but it is not equality-free.

  (The reader concerned about having a proper class of variables may instead consider an infinite set of variables of some cardinality $\kappa$ and add the following constraint on filter classes and logical classes: if each $\kappa$-generated substructure of $\pair{\alg{A}}{F}$ belongs to~$\class{K}$, then so does $\pair{\alg{A}}{F}$. With this modification, logical classes are precisely the classes of all models of some logic over the given set of variables in the sense of abstract algebraic logic~\cite{font16}. The problem of describing the logical classes of $\class{L}$-structures is therefore intimately related to the problem of describing the extensions of the logic of all $\class{L}$-structures.)

  We now introduce some notational conventions for talking about implications. When no confusion is likely to arise, we identify the atomic formula $\Fsymbol(t)$ with the term $t$. Atomic formulas, or the corresponding terms, will be denoted by lowercase Greek letters. A set of atomic formulas of the form $\Fsymbol(t)$, or the corresponding set of terms, will be denoted by $\Gamma$, $\Delta$, or $\Phi$, while a set of equations will be denoted by~$\Eps$. We adopt notation common in abstract algebraic logic and write $\Eps, \Gamma \vdash \varphi$ instead of $\Eps ~ \& ~ \Gamma \implies \varphi$ and $\Eps, \Gamma \vdash t \equals u$ instead of $\Eps ~ \& ~ \Gamma \implies t \equals u$. For example, the equality-free implication which defines $2$-filters becomes simply $x \wedge y, y \wedge z, z \wedge x \vdash x \wedge y \wedge z$ in this notation. The notation $\Eps, \Gamma \vdash_{\class{K}} \varphi$ means that the filter implication $\Eps, \Gamma \vdash \varphi$ holds in each structure in $\class{K}$.

  The following syntactic description of filter classes and logical classes is due to Stronkowski~\cite[Lemma~2.3 and Theorems~3.7 \& 3.8]{stronkowski18}. The claim for finitary logical classes, as well as its infinitary version in case we restrict to a set of variables, is originally due to Dellunde and Jansana~\cite{dellunde+jansana96}.

\begin{theorem} \label{thm: axiomatization by implications}
  Let $\class{K}$ be a class of $\class{L}$-structures. Then
\begin{enumerate}[(i)]
\item $\class{K}$ is a (finitary) filter class if and only if $\class{K}$ is axiomatized by some class of (finitary) filter implications.
\item $\class{K}$ is a (finitary) logical class if and only if $\class{K}$ is axiomatized by some class of (finitary) equality-free implications.
\end{enumerate}
\end{theorem}

  We shall be interested in cases where the quasivariety $\class{L}$ is the class $\DLat$ of distributive lattices, the class $\SLat$ of meet semilattices, the class $\uSLat$ of unital meet semilattices, and the class $\BAlg$ of Boolean algebras. We do not explicitly consider bounded distributive lattices and bounded meet semilattices, but our results easily extend to these classes. (Recall that a \emph{unital} meet semilattice is an upper bounded meet semilattice expanded by a constant $\True$ which denotes the top element.)

  Distributive lattices equipped with an upset form a filter class $\DLclass{\infty}$ of $\DLat$-structures. Distributive lattices equipped with an $n$-filter then form a filter subclass of $\DLclass{\infty}$ denoted $\DLclass{n}$ for $n \in \omega$. Clearly $\DLclass{i} \subseteq \DLclass{j} \subseteq \DLclass{\infty}$ for $i < j$, and the structures $\BAm{n}$ witness that these inclusions are strict. The classes of $\SLat$-structures $\SLclass{\infty}$ and $\SLclass{n}$ are defined in the same way. The classes of $\uSLat$-structures $\uSLclass{\infty}$ and $\uSLclass{n}$ for $n \geq 1$, and the classes of $\BAlg$-structures $\BAclass{\infty}$ and $\BAclass{n}$ for $n \geq 1$ are defined similarly, but we restrict to non-empty upsets and non-empty $n$-filters.

  The filter implications valid in $\DLclass{n}$ can be described in terms of implications valid in $\DLclass{\infty}$, which in turn can be described in terms of implications valid in distributive lattices. We write $\Eps \vdash_{\DLat} \gamma \leq \varphi$ if the implication $\Eps \vdash \gamma \wedge \varphi \equals \gamma$ holds in all distributive lattices. An entirely analogous theorem can be proved relating the filter classes $\SLclass{n}$ and $\SLclass{\infty}$ and implications valid in meet semilattices.

\begin{theorem}
  Let $\Eps, \Gamma \vdash \varphi$ be a filter implication in the signature of distributive lattices. Then:
\begin{enumerate}[(i)]
\item $\Eps, \gamma \vdash_{\DLclass{\infty}} \varphi$ if and only if $\Eps \vdash_{\DLat} \gamma \leq \varphi$.
\item $\Eps, \Gamma \vdash_{\DLclass{\infty}} \varphi$ if and only if $\Eps, \gamma \vdash_{\DLclass{\infty}} \varphi$ for some $\gamma \in \Gamma$.
\item $\Eps, \Gamma \vdash_{\DLclass{n}} \varphi$ if and only if there is some non-empty finite set of terms $\Phi$ such that $\Eps, \Gamma \vdash_{\DLclass{\infty}} \bigwedge \Delta$ for each $\Delta \bsubseteq{n} \Phi$ and $\Eps, \bigwedge \Phi \vdash_{\DLclass{\infty}} \varphi$.
\end{enumerate}
\end{theorem}

\begin{proof}
  The first equivalence is trivial. The third equivalence follows from the $n$-filter generation lemma (Lemma~\ref{lemma: generating n filters}) applied to the quotient $\FreeDLat{X} / \theta_{\Eps}$ of the free distributive lattice $\FreeDLat{X}$ generated by the set $X$ of all variables which occur in $\Gamma \cup \{ \varphi \}$ by the congruence $\theta_{\Eps}$ generated by $\Eps$. To prove the second equivalence, let $\FmAlg$ be the absolutely free algebra (the algebra of terms) over the set of all variables which occur in $\Gamma \cup \{ \varphi \}$ and suppose that ${\Eps, \gamma \nvdash_{\DLclass{\infty}} \varphi}$ for each $\gamma \in \Gamma$. This yields a distributive lattice $\alg{L}_{\gamma}$ and a homo\-morphism $h_{\gamma}\colon \FmAlg \to \alg{L}_{\gamma}$ such that $h_{\gamma}(\gamma) \nleq h(\varphi)$ for each $\gamma \in \Gamma$ and $h_{\gamma}(t) = h_{\gamma}(u)$ for each equality $t \equals u$ in $\Eps$. Let $\alg{L} \assign \prod_{\gamma \in \Gamma} \alg{L}_{\gamma}$, let $h\colon \FmAlg \to \alg{L}$ be the product of the maps $h_{\gamma}$, and let $F$ be the upset of $\alg{L}$ generated by $\set{h(\gamma)}{\gamma \in \Gamma}$. Clearly $h(\gamma) \nleq h(\varphi)$ for each $\gamma \in \Gamma$, hence $h(\varphi) \notin F$. Also, $h(t) = h(u)$ for each equality $t \equals u$ in $\Eps$. The structure $\pair{\alg{L}}{F}$ thus witnesses that $\Eps, \Gamma \nvdash_{\DLclass{\infty}} \varphi$. The right-to-left implication is trivial.
\end{proof}

  For Boolean algebras and unital meet semilattices, the above equivalences need to be modified slightly. Again, an entirely analogous fact can be proved relating the filter classes $\uSLclass{n}$ and $\uSLclass{\infty}$ and implications valid in unital meet semilattices.

\begin{theorem}
  Let $\Eps, \Gamma \vdash \varphi$ be a filter implication in the signature of Boolean algebras. Then:
\begin{enumerate}[(i)]
\item $\Eps \vdash_{\BAclass{\infty}} \varphi$ if and only if $\Eps \vdash_{\BAlg} \True \leq \varphi$.
\item $\Eps, \gamma \vdash_{\BAclass{\infty}} \varphi$ if and only if $\Eps \vdash_{\BAlg} \gamma \leq \varphi$.
\item $\Eps, \Gamma \vdash_{\BAclass{\infty}} \varphi$ for $\Gamma$ non-empty if and only if $\Eps, \gamma \vdash_{\BAclass{\infty}} \varphi$ for some $\gamma \in \Gamma$.
\item $\Eps, \Gamma \vdash_{\BAclass{n}} \varphi$ if and only if there is some non-empty finite set of terms $\Phi$ such that $\Eps, \Gamma \vdash_{\BAclass{\infty}} \bigwedge \Delta$ for each $\Delta \bsubseteq{n} \Phi$ and $\Eps, \bigwedge \Phi \vdash_{\BAclass{\infty}} \varphi$.
\end{enumerate}
\end{theorem}

  In order to understand what happens beyond the finitary case, it remains to introduce one more family of upsets. An upset of a \mbox{distributive} lattice (Boolean algebra, semilattice, unital semilattice) will be called \emph{residually finite} if it is an inter\-section of homomorphic preimages of some upsets of finite distributive lattices (Boolean algebras, semilattices, unital semilattices). Observe that an upset of a Boolean algebra $\alg{A}$ is residually finite if and only if it is residually finite as an upset of the distributive lattice reduct of $\alg{A}$: each homomorphism from a Boolean algebra $\alg{A}$ to a finite distributive lattice $\alg{B}$ extends to a homomorphism from $\alg{A}$ to the free Boolean extension of $\alg{B}$, which is also finite. Similarly, an upset of a unital semilattice $\alg{S}$ is residually finite if and only if it is residually finite as an upset of the semilattice reduct of $\alg{S}$. Moreover, an upset of a distributive lattice is residually finite if and only if it is an intersection of a family of upsets which are $n$-filters for some $n$ (not necessarily the same $n$ for the whole family).

  The class of all \mbox{distributive} lattices (Boolean algebras, semilattices, unital semilattices) equipped with a residually finite upset is a filter subclass of $\DLclass{\infty}$ denoted~$\DLclass{\omega}$ ($\BAclass{\omega}$, $\SLclass{\omega}$, $\uSLclass{\omega}$). Each $n$-filter on a distributive lattice is a residually finite upset but some residually finite upsets, such as the upset of $\prod_{n \geq 1} \BAm{n}$, are not $n$-filters for any $n$. The class $\DLclass{\omega}$ is a proper subclass of $\DLclass{\infty}$: consider for example the upset $\nonempty{\omega}$ of non-zero elements of $\BA{\omega} \assign (\BA{1})^{\omega}$. It follows that there is some filter implication which holds in each finite structure in $\DLclass{\infty}$ (or equivalently, in each upset which is an $n$-filter for some $n$) but which fails in some infinite one. Indeed, one such implication states that if $X_{1}, X_{2}, \dots$ are sets consisting of at most $2, 3, \dots$ elements in a distributive lattice such that $\bigwedge X_{1} = \bigwedge X_{2} = \dots = a$ and the meet of each proper subset of $X_{1}, X_{2}, \dots$ lies in $F$, then so does $a$. This implication holds for each upset which is an $n$-filter for some~$n$, but it fails in $\pair{\BA{\omega}}{\nonempty{\omega}}$.

\section{Filter classes of upsets of distributive lattices}

  We now show that the only finitary filter subclasses of $\DLclass{\infty}$ are the classes $\DLclass{n}$. The following theorem merely restates the fact that $n$-filters on distributive lattices are precisely the inter\-sections of homomorphic preimages of $\nonempty{n} \subseteq \BA{n}$. Recall that we are using the notation $\BAm{n} \assign \pair{\BA{n}}{\nonempty{n}}$. For $n = 0$ we take $\BA{0}$ to be the singleton Boolean lattice and $\nonempty{0} \assign \emptyset$.

\begin{theorem}
  $\DLclass{n}$ is generated as a filter class by the structure $\BAm{n}$.
\end{theorem}

\begin{theorem}
  $\DLclass{\omega}$ is generated as a filter class by $\set{\BAm{n}}{n \geq 1}$.
\end{theorem}

   An intruiging problem which we leave open is how to axiomatize $\DLclass{\omega}$.

\begin{theorem}
  $\DLclass{\infty}$ is generated as a finitary filter class by $\set{\BAm{n}}{n \geq 1}$.
\end{theorem}

\begin{proof}
  Each structure embeds into an ultraproduct of its finitely generated substructures and each finitely generated structure in $\DLclass{\infty}$ is finite. Each finite structure $\pair{\alg{L}}{F}$ in $\DLclass{\infty}$ then embeds into a product of the structures $\BAm{n}$, since $F$ is an intersection of prime upsets of $\alg{L}$ and each prime upset of $\alg{L}$ is a prime $n$-filter for some $n$ because $\alg{L}$ is finite.
\end{proof}

\begin{theorem}
  $\DLclass{\infty}$ is generated as a logical class by $\set{\BAm{n}}{n \geq 1}$.
\end{theorem}

\begin{proof}
  We first define an auxiliary construction on distributive lattice terms. Recall that a distributive lattice term in \emph{disjunctive normal form} is a disjunction of conjunctions of variables. If $t$ is such a term and $\tuple{x}$ is a tuple of variables, we define the term $t[\tuple{x}]$ as follows: if one of the conjunctions in $t$ consists entirely of variables not in $\tuple{x}$, we leave $t[\tuple{x}]$ undefined, otherwise we remove all variables not in $\tuple{x}$ from $t$. If $t[\tuple{x}]$ is defined, then the inequality $t \leq t[\tuple{x}]$ holds in each distributive lattice. If $\Gamma$ is a set of terms in disjunctive normal form, we define $\Gamma[\tuple{x}]$ as the set of all $t[\tuple{x}]$ such that $t \in \Gamma$ and $t[\tuple{x}]$ is defined.

  Now consider an equality-free implication $\Gamma \vdash u(\tuple{x})$ which holds in each $\BAm{n}$ for $n \geq 1$. We may assume without loss of generality that each term in $\Gamma$ is in disjunctive normal form, i.e.\ a disjunction of conjunctions of variables. If $\Gamma \vdash u(\tuple{x})$ holds in $\BAm{n}$ for each $n \geq 1$, then in particular it holds if we assign the value $\True$ to every variable outside of $\tuple{x}$. It follows that $\Gamma[\tuple{x}] \vdash u(\tuple{x})$ holds in each $\BAm{n}$ for $n \geq 1$. But $\Gamma[\tuple{x}]$ only contains finitely many variables, therefore up to equivalence of distributive lattice terms $\Gamma[\tuple{x}]$ is finite and $\Gamma[\tuple{x}] \vdash u(\tuple{x})$ is a finitary equality-free implication. It follows that $\Gamma[\tuple{x}] \vdash u(\tuple{x})$ holds in each structure in $\DLclass{\infty}$. Finally, $\Gamma \vdash u(\tuple{x})$ follows from $\Gamma[\tuple{x}] \vdash u(\tuple{x})$ because $t \vdash t[\tuple{x}]$ holds in $\DLclass{\infty}$ whenever $t[\tuple{x}]$ is defined.
\end{proof}

  We can state this is more explicit terms.

\begin{theorem}
  Every upset of a distributive lattice is a strict homomorphic image of an intersection of homomorphic preimages of finitely generated upsets.
\end{theorem}

  The free distributive lattice over a set~$X$ will be denoted $\FreeDLat{X}$ in the following, with $\FreeDLat{n} \assign \FreeDLat{\{ x_{1}, \dots, x_{n} \}}$. Observe that $x_{1} \wedge \dots \wedge x_{n}$ is the bottom element of $\FreeDLat{n}$. Accordingly, $P(\FreeDLat{n}) \assign \set{a \in \FreeDLat{n}}{a > x_{1} \wedge \dots \wedge x_{n}}$.

\begin{lemma}
  The structure $\BAm{n}$ embeds into $\pair{\FreeDLat{n}}{P(\FreeDLat{n})}$.
\end{lemma}

\begin{proof}
  $\FreeDLat{n}$ is isomorphic to the lattice of non-empty non-total subsets of the poset $\BA{n}$. Each preorder on this poset which extends the partial order of $\BA{n}$ determines a subalgebra of $\FreeDLat{n}$ and thus a substructure of $\pair{\FreeDLat{n}}{P(\FreeDLat{n})}$. The preorder which collapses all the elements strictly below the coatoms into a single point then determines a substructure isomorphic to $\BAm{n}$.
\end{proof}

\begin{theorem}[Filter subclasses of $\DLclass{\infty}$] \label{thm: extensions of dlinfty}
  Let $\class{K}$ be a non-trivial filter subclass of $\DLclass{\infty}$. Then either $\DLclass{\omega} \subseteq \class{K}$ or $\class{K}$ is one of the classes $\DLclass{n}$ for $n \in \omega$.
\end{theorem}

\begin{proof}
  Because $\class{K}$ is non-trivial, it contains some $\pair{\alg{L}}{F}$ where $F$ is a proper filter. Then $\BAm{0}$ is a substructure of $\pair{\alg{L}}{F}$, so $\BAm{0} \in \class{K}$ and $\DLclass{0} \subseteq \class{K}$. If $\class{K}$ contains some $\pair{\alg{L}}{F}$ where $F$ is a non-empty proper filter, then $\BAm{1}$ is a substructure of $\pair{\alg{L}}{F}$, so $\BAm{1} \in \class{K}$ and $\DLclass{1} \subseteq \class{K}$. Otherwise $\class{K} = \DLclass{0}$.

  Now suppose that $\class{K} \nsubseteq \DLclass{n}$ for some $n \geq 1$. We show that $\DLclass{n+1} \subseteq \class{K}$. The theorem then follows at once, taking into account that $\DLclass{\omega} \subseteq \class{K}$ if $\BAm{n} \in \class{K}$ for each $n \geq 1$. Because $\class{K} \nsubseteq \DLclass{n}$, there is some $\pair{\alg{L}}{F} \in \class{K}$ such that $F$ is not an $n$-filter. Since $n$-filters are defined by an implication involving only $n+1$ variables, there is an $(n+1)$-generated subalgebra of $\pair{\alg{L}}{F}$ such that the restriction of $F$ to this subalgebra is not an $n$-filter. We~may therefore assume without loss of generality that $\alg{L}$ is $(n+1)$-generated. Then there is a surjective homomorphism $h\colon \FreeDLat{n+1} \to \alg{L}$ such that $h(\bigwedge_{j \neq i} x_{j}) \in F$ but $h(x_{1} \wedge \dots \wedge x_{n+1}) \notin F$. The structure $\pair{\FreeDLat{n+1}}{h^{-1}[F]}$ then lies in $\class{K}$. Since each element of $\FreeDLat{n+1}$ is a join of meets of subsets of $\{ x_{1}, \dots, x_{n+1} \}$ and each meet of a proper subset of $\{ x_{1}, \dots, x_{n+1} \}$ lies in $h^{-1}[F]$, it follows that each element of $\FreeDLat{n+1}$ other than $x_{1} \wedge \dots \wedge x_{n+1}$ lies in $h^{-1}[F]$. In other words, $h^{-1}[F] = P(\FreeDLat{n+1})$. But then $\BAm{n+1}$ is isomorphic to a substructure of $\pair{\FreeDLat{n+1}}{h^{-1}[F]}$ and therefore $\BAm{n+1} \in \class{K}$ and $\DLclass{n+1} \subseteq \class{K}$.
\end{proof}

  The above theorem holds equally well for lower bounded, upper bounded, and bounded distributive lattices (excluding $n = 0$ in the last two cases).

\begin{corollary}[Finitary filter subclasses of $\DLclass{\infty}$] \label{cor: extensions of dlinfty}
  The only non-trivial proper finitary filter subclasses of $\DLclass{\infty}$ are the classes $\DLclass{n}$ for $n \in \omega$.
\end{corollary}

\begin{corollary}[Logical subclasses of $\DLclass{\infty}$]
  The only non-trivial proper logical subclasses of $\DLclass{\infty}$ are the classes $\DLclass{n}$ for $n \in \omega$.
\end{corollary}

\begin{corollary}
  If $F$ is an $n$-filter but not an $(n+1)$-filter on a distributive lattice $\alg{L}$, then $\pair{\alg{L}}{F}$ generates $\DLclass{n}$ as a filter class.
\end{corollary}

\begin{corollary}
  If $F$ is an upset but not an $n$-filter for any $n \in \omega$ on a distributive lattice $\alg{L}$, then $\pair{\alg{L}}{F}$ generates $\DLclass{\infty}$ as a finitary filter class and as a logical class.
\end{corollary}

  The above corollary allows us to determine the finitary filter class generated by any structure of the form $\pair{\alg{L}}{F}$ where $\alg{L}$ is a distributive lattice and $F$ is an upset of $\alg{L}$. For example, consider the upset~$\height{d}{m}$ of all elements $a \in \BA{n}$ for $n \assign d + m$ of height strictly more than $m$. Equivalently, this is the upset of all $a \in \BA{n}$ of co-height strictly less than $d$. Recall that the height (co-height) of an element $a$ in a poset is the size of the longest chain in the principal downset (upset) generated by $a$, minus one. Equivalently, the height of $a \in \BA{n}$ is the number of atoms below~$a$, and the co-height of $a$ is the number of co-atoms above $a$. This convention fits well with our previous notation: $\nonempty{n} =\height{n}{0}$.

\begin{fact} \label{fact: height separation}
  $\height{d}{m} \subseteq \BA{d+m}$ is a $d$-filter whose complement is an $(m+1)$-ideal. It~is not a $(d-1)$-filter and its complement is not an $m$-ideal.
\end{fact}

\begin{proof}
  $\height{d}{m}$ is a $d$-filter by virtue of being the intersection of a family of $d$-filters. This family consists of all unions of $d$ principal filters generated by atoms of~$\BA{d+m}$. (If $Y \subseteq X$ with $\card{X} = n$, then $\card{Y} > m$ if and only if $Y$ intersects each subset of $X$ of cardinality $d = n - m$.) It is not a $(d - 1)$-filter because the meet of each set of at most $d-1$ coatoms has co-height at most $d - 1$, but the meet of all coatoms is~$\False$. The other claims follow if we apply the same reasoning to $\BA{d+m} \setminus \height{d}{m}$ instead of~$\height{d}{m}$.
\end{proof}

\begin{fact}
  $\pair{\BA{n+m}}{\height{n}{m}}$ generates $\DLclass{n}$ as a filter class for each $m \in \omega$.
\end{fact}

\section{Filter classes of upsets of meet semilattices}

  We have already seen that not every $n$-filter on a meet semilattice is an inter\-section of homomorphic preimages of the $n$-filter $\nonempty{n}$. This is of course also witnessed syntactically. The filter implication
\begin{align*}
  x \wedge y \equals y \wedge z \equals z \wedge x ~ \& ~ \Fsymbol(x) ~ \& ~ \Fsymbol(y) ~ \& ~ \Fsymbol(z) \implies \Fsymbol(x \wedge y \wedge z)
\end{align*}
  holds in $\BAm{2}$ but not in $\pair{\Mfive}{P(\Mfive)}$, where $\Mfive$ is the five-element diamond. Nonetheless, we show that every $n$-filter is the strict homomorphic image of such an intersection. In other words, $\SLclass{n}$ and $\uSLclass{n}$ are generated by $\BAm{n}$ as logical classes but (unlike in the distributive lattice case) not as filter classes.

  The key observation in this respect is that the free unital meet semilattice $\FreeuSLat{X}$ over a set of generators $X$ is in fact a distributive lattice.

\begin{theorem}
  $\uSLclass{n}$ is generated as a logical class by the structure $\BAm{n}$.
\end{theorem}

\begin{proof}
  Let $F$ be an $n$-filter on a unital semilattice $\alg{S}$. Then there is a surjective homo\-morphism $h\colon \FreeuSLat{X} \to \alg{S}$ for some sufficiently large $X$. Let $G \assign h^{-1}[F]$. Then $h\colon \pair{\FreeuSLat{X}}{G} \to \pair{\alg{S}}{F}$ is a strict surjective homomorphism. Because $\FreeuSLat{X}$ is a distributive lattice and $G$ is an $n$-filter, $G$ is an intersection of homo\-morphic preimages of $\nonempty{n}$. Moreover, these homomorphisms preserve the top element of $\FreeuSLat{X}$, i.e.\ they are unital semilattice homomorphisms.
\end{proof}

\begin{theorem}
  $\uSLclass{\infty}$ is generated as a logical class by $\set{\BAm{n}}{n \geq 1}$.
\end{theorem}

\begin{proof}
  The proof is entirely analogous to the distributive lattice case.
\end{proof}

  The last two theorems are equally true for $\SLclass{n}$ and $\SLclass{\infty}$. The classes $\SLclass{n}$ were in fact already considered by Font and Moraschini~\cite{font+moraschini14}, who showed that they form an infinite increasing chain. More precisely, they studied the corresponding logics, which they denoted $\mathcal{R}_{n}$.

\begin{theorem}[Logical subclasses of $\uSLclass{\infty}$]
  Let $\class{K}$ be a non-trivial logical subclass of $\uSLclass{\infty}$. Then either $\uSLclass{\omega} \subseteq \class{K}$ or $\class{K}$ is one of the classes $\uSLclass{n}$ for $n \geq 1$.
\end{theorem}

\begin{proof}
  It~suffices to replace $\FreeDLat{n+1}$ in the proof of Theorem~\ref{thm: extensions of dlinfty} by $\FreeuSLat{n+1}$ and to observe that $\FreeuSLat{n+1}$ is the unital meet semilattice reduct of $\BA{n+1}$.
\end{proof}

\begin{corollary}[Finitary logical subclasses of $\uSLclass{\infty}$] \label{cor: extensions of uslinfty}
  The only non-trivial proper finitary logical subclasses of $\uSLclass{\infty}$ are the classes $\uSLclass{n}$ for~$n \geq 1$.
\end{corollary}

  The class $\SLclass{\infty}$, on the other hand, admits other finitary logical subclasses in addition to the classes $\SLclass{n}$. For example, the implication
\begin{align*}
  x \wedge z, y \wedge z & \vdash x \wedge y \wedge z
\end{align*}
 fails in $\BAm{2}$ but it holds in the three-element substructure obtained by removing the top element. This implication is obtained from ${x, y \vdash x \wedge y}$ using the substitution $x \mapsto x \wedge z$ and $y \mapsto y \wedge z$. An upset $F$ of a semilattice therefore validates this implication if and only if $a \wedge b \in F$ holds whenever $a \in F$ and $b \in F$ and $a$ and $b$ have an upper bound. Observe that if the constant~$\True$ is part of the signature, then the above implication defines ordinary filters: it~suffices to substitute $\True$ for $z$.

  Many other non-equivalent conditions may be obtained by applying such sub\-stitutions to \emph{$n$-adjunction}, i.e.\ the implication which defines $n$-filters:
\begin{align*}
  \set{\bigwedge_{j \neq i} x_{j}}{i \in \range{n+1}} \vdash x_1 \wedge \dots \wedge x_{n+1}.
\end{align*}
  For example, consider the following substitution applied to $5$-adjunction:  
\begin{align*}
  & x_{i} \mapsto x_i \wedge y \text{ for $i \in \{ 1, 2, 3 \}$}, & & x_{j} \mapsto x_j \wedge z \text{ for $j \in \{ 4, 5 \}$}, & & x_{6} \mapsto x_6.
\end{align*}
  The corresponding implication states that the meet of each $6$-tuple of elements lies in $F$ provided that (i) each non-empty proper submeet lies in $F$, (ii) three of the six elements have an upper bound, and (iii) so do two of the other three elements. We conjecture that each finitary filter subclass of $\SLclass{\infty}$ is axiomatized by such substitution instances of $n$-adjunction.

\section{Filter classes of upsets of Boolean algebras}
\label{sec: filter classes of bas}

  The picture for Boolean algebras is substantially more complicated: there are many finitary filter subclasses of $\BAclass{\infty}$ in addition to the classes $\BAclass{n}$ for $n \geq 1$. Nonetheless, the classes $\BAclass{n}$ still play an important role: they are the only ones generated by \emph{prime} structures, i.e.\ structures of the form $\pair{\alg{A}}{F}$ where $F$ is a prime upset. Moreover, as in the case of distributive lattices, each finitary filter subclass of $\BAclass{\infty}$ is in fact a logical class.

  The proofs of the following theorems carry over from distributive lattices.

\begin{theorem}
  $\BAclass{n}$ is generated as a filter class by the structure $\BAm{n}$.
\end{theorem}

\begin{theorem}
  $\BAclass{\omega}$ is generated as a filter class by $\set{\BAm{n}}{n \geq 1}$.
\end{theorem}

  We again leave the problem of axiomatizing $\BAclass{\omega}$ open. We do not even know whether $\BAclass{\omega}$ is a logical class, or more specifically whether it coincides with the logical class generated by the structures $\BAm{n}$.

\begin{theorem}
  $\BAclass{\infty}$ is generated as a finitary filter class by $\set{\BAm{n}}{n \geq 1}$.
\end{theorem}

  It is not the case (as it was for distributive lattices) that $\BAclass{\infty}$ is generated as a logical class by $\set{\BAm{n}}{n \geq 1}$. For example, each of the structures $\BAm{n}$ satisfies the following infinitary equality-free implications:
\begin{align*}
  \set{(x_{i} \wedge \neg x_{j}) \vee y}{i, j \in \omega \text{ and } i < j} & \vdash y, \\
  \set{(x_{i} \wedge \neg x_{j}) \vee y}{i, j \in \omega \text{ and } i > j} & \vdash y.
\end{align*} 
  These express the fact that the algebras $\BA{n}$ are well-partially-ordered and dually well-partially-ordered. Recall that a (dual) well-partial-order is a partial order such that in each sequence $x_{i}$ for $i \in \omega$ there are $i < j$ with $x_{i} \leq x_{j}$ (with $x_{i} \geq x_{j}$). But these implications fail in $\pair{\BA{\omega}}{\nonempty{\omega}}$.

  The key to proving that the only finitary filter subclasses of $\BAclass{\infty}$ generated by prime structures are the classes $\BAclass{n}$ is to find suitable splittings of the lattice of all filter subclasses of $\BAclass{\infty}$. In particular, we identify finitary equality-free implications $(\alpha_{n})$ such that for each filter subclass $\class{K}$ of $\BAclass{\infty}$ either $(\alpha_{n})$ holds in $\class{K}$ or $\BAclass{n} \subseteq \class{K}$, but not both. We now describe these implications.

  Let us pick $k \geq 1$ so that $2^{k} \geq n$. There are $2^{k}$ complete conjunctive clauses over $k$ variables $x_{1}, \dots, x_{k}$ up to equivalence in Boolean algebras. These are conjunctions of variables and negated variables such that for each of the $k$ variables $x_{i}$ exactly one element of $\{ x_{i}, \neg x_{i} \}$ occurs in the conjunction. Let us call these complete conjunctive clauses $\pi_{i}$ for~$i \in \range{2^{k}}$. Now consider the implication
\begin{align*}
  \pi_{1}, \dots, \pi_{n} \vdash \pi_{n+1} \vee \dots \vee \pi_{2^{k}}. \tag{$\alpha_{n}$}
\end{align*}
  In case $n = 2^{k}$, we interpret this simply as $\pi_{1}, \dots, \pi_{n} \vdash y$. In particular, we may take $(\alpha_{2})$ to be the rule $x, \neg x \vdash y$. We take $(\alpha_{1})$ to be $x \vdash y$.

\begin{lemma}
  Let $\alg{A}$ be a finite Boolean algebra and $F$ be a \emph{prime} upset of $\alg{A}$. If $F$ is an $n$-filter but not an $(n-1)$-filter, then $\BAm{n}$ embeds into $\pair{\alg{A}}{F}$.
\end{lemma}

\begin{proof}
  Each finite Boolean algebra is isomorphic to the powerset of some finite set $X$. Each prime $n$-filter $F$ on this powerset which is not an $(n-1)$-filter is then determined by an $n$-element set $\{ x_{1}, \dots, x_{n} \} \subseteq X$ as follows: $U \in F$ if and and only if $x_{i} \in U$ for some $x_{i}$. Each equivalence relation on $X$ determines a subalgebra of the powerset (generated by the equivalence classes). Any equivalence relation where each equivalence class contains exactly one of the elements $x_{i}$ then determines a substructure isomorphic to $\BAm{n}$.
\end{proof}

\begin{lemma}
  Let $\alg{A}$ be a Boolean algebra and $F$ be an upset of $\alg{A}$. If $(\alpha_{n})$ fails in the structure $\pair{\alg{A}}{F}$, then $\BAm{n}$ embeds into $\pair{\alg{A}}{F}$.
\end{lemma}

\begin{proof}
  Suppose that $(\alpha_{n+1})$ fails in $\pair{\alg{A}}{F}$. This is witnessed by a homomorphism $h\colon \FreeBA{k} \to \alg{A}$ from the free Boolean algebra over the generators $x_{1}, \dots, x_{k}$. Let the substructure $\pair{\alg{B}}{G}$ of $\pair{\alg{A}}{F}$ be the image of~$h$ and let $H \assign h^{-1}[G]$. Then $h\colon \pair{\FreeBA{k}}{H} \to \pair{\alg{B}}{G}$ is a strict surjective homo\-morphism. We have $\pi_{i} \in H$ for each $i \in \range{n}$, while $\pi_{n+1} \vee \dots \vee \pi_{2^{k}} \notin H$. Here we identify the terms $\pi_{i}$ with the atoms of the free Boolean algebra $\FreeBA{k}$. This uniquely determines $H$: each element of $\FreeBA{k}$ either lies below $\pi_{n+1} \vee \dots \vee \pi_{2^{k}}$ or above some $\pi_{i}$ for $i \in \range{n}$.

  Let $I \assign \FreeBA{k} \setminus H$. Then $I \subseteq \FreeBA{k}$ is the principal ideal generated by the element $\neg (\pi_{n+1} \wedge \dots \wedge \pi_{2^{k}})$. The projection map $\FreeBA{k} \to \FreeBA{k} / I$ is a strict homomorphism of structures $\pair{\FreeBA{k}}{H} \to \pair{\FreeBA{k} / I}{H / I}$, where we use the notation $H / I \assign \set{h / I}{h \in H}$. Clearly $H / I$ consists precisely of the non-zero elements of $\FreeBA{k} / I$. Moreover, $\FreeBA{k} / I$ has exactly $n$ atoms, namely $\pi_{1} / I, \dots, \pi_{n} / I$. The structure $\pair{\FreeBA{k} / I}{H / I}$ is thus isomorphic to $\BAm{n}$. It follows that $H$, hence also $G$, is an $n$-filter which is not an $(n-1)$-filter. But $\pair{\alg{B}}{G}$ is a finite structure, therefore $\BAm{n}$ is isomorphic to a substructure of $\pair{\alg{B}}{G}$.
\end{proof}

\begin{lemma}
  Let $\alg{A}$ be a Boolean algebra and $F$ be a prime upset of $\alg{A}$. If $(\alpha_{n+1})$ holds in $\pair{\alg{A}}{F}$, then $F$ is an $n$-filter.
\end{lemma}

\begin{proof}
  Suppose that $F$ is not an $n$-filter. We show that $(\alpha_{n+1})$ fails in $\pair{\alg{A}}{F}$. There are $a_{1}, \dots, a_{n+1}$ such that $\bigwedge_{j \neq i} a_{j} \in F$ for each $i \in \range{n+1}$ but $\bigwedge_{i \in \range{n+1}} a_{i} \notin F$. Let $\alg{B}$ be the subalgebra of $\alg{A}$ generated by $a_{1}, \dots, a_{n+1}$ and let $G$ be the restriction of $F$ to $\alg{B}$. Then $\alg{B}$ is finite, therefore $G$ is a prime $m$-filter for some $m > n$. (Recall that $F$ is prime.) It follows that there is a strict homomorphism $h\colon \pair{\alg{B}}{G} \to \BAm{m}$. The image of $h$ is some substructure of $\BAm{m}$ whose upset is not an $n$-filter. This image is therefore isomorphic to $\BAm{k}$ for some $k > n$. In~other words, we have a strict surjective homomorphism from $\pair{\alg{B}}{G}$ to $\BAm{k}$ for some $k > n$. But $(\alpha_{n+1})$ fails in $\BAm{k}$ (because $\BAm{n+1}$ is a substructure of $\BAm{k}$), hence also in its strict homomorphic preimage $\pair{\alg{B}}{G}$ and in the structure $\pair{\alg{A}}{F}$.
\end{proof}

\begin{lemma} \label{lemma: cl n split}
  Let $\class{K}$ be a filter subclass of $\BAclass{\infty}$. Then either $\BAclass{n} \subseteq \class{K}$ or $(\alpha_{n})$ holds in $\class{K}$, but not both.
\end{lemma}

\begin{proof}
  The implication $(\alpha_{n})$ fails in $\BAm{n}$ because the upset~$\nonempty{n}$ is not an $(n-1)$-filter, therefore $\BAclass{n} \subseteq \class{K}$ implies that $(\alpha_{n})$ does not hold in $\class{K}$. On the other hand, if $(\alpha_{n})$ fails in some structure $\pair{\alg{A}}{F} \in \class{K}$, then $\BAm{n} \in \class{K}$ by virtue of being isomorphic to a substructure of $\pair{\alg{A}}{F}$. But then $\BAclass{n} \subseteq \class{K}$.
\end{proof}

\begin{fact}
  Each proper filter subclass of $\BAclass{\omega}$ validates some $(\alpha_{n})$.
\end{fact}

\begin{proof}
  If for each $n$ the filter subclass $\class{K}$ fails to validate $(\alpha_{n})$, then $\BAm{n} \in \BAclass{n} \subseteq \class{K}$ for each $n \geq 1$, so $\BAclass{\omega} \subseteq \class{K}$.
\end{proof}

\begin{fact}
  Each proper finitary filter subclass of $\BAclass{\infty}$ validates some $(\alpha_{n})$.
\end{fact}

\begin{proof}
  If $\class{K}$ is finitary, then $\BAclass{\omega} \subseteq \class{K}$ implies $\BAclass{\infty} \subseteq \class{K}$.
\end{proof} 

\begin{theorem}[Filter subclasses of $\BAclass{\omega}$]
  Let $\class{K}$ be a non-trivial filter subclass of $\BAclass{\omega}$ generated by prime structures. Then either $\BAclass{\omega} \subseteq \class{K}$ or $\class{K}$ is one of the classes $\BAclass{n}$ for $n \geq 1$.
\end{theorem}

\begin{proof}
  If $\BAclass{\omega} \nsubseteq \class{K}$, then $\BAm{m} \notin \class{K}$ for some $m \geq 1$, i.e.\ $\BAclass{m} \nsubseteq \class{K}$. On the other hand, $\class{K}$ contains some structure $\pair{\alg{L}}{F}$ where $F$ is not the total filter. Then $\BAm{1}$ is a substructure of $\pair{\alg{L}}{F}$ and $\BAclass{1} \subseteq \class{K}$. There is therefore some $n \in \omega$ such that $\BAclass{n} \subseteq \class{K}$ but $\BAclass{n+1} \nsubseteq \class{K}$. But then $(\alpha_{n+1})$ holds in $\class{K}$, so $\class{K} \subseteq \BAclass{n}$.
\end{proof}

\begin{theorem}[Logical subclasses of $\BAclass{\infty}$]
  The only non-trivial proper finitary logical subclasses of $\BAclass{\infty}$ generated by prime structures are $\BAclass{n}$ for~$n \geq 1$.
\end{theorem}

\begin{proof}
  It suffices to replace filter classes by logical classes and $\BAclass{\omega}$ by $\BAclass{\infty}$ in the previous proof.
\end{proof}

  In the rest of the paper, we consider some further filter subclasses of $\BAclass{\infty}$. One such family of subclasses of $\BAclass{\infty}$ may be defined by the implications
\begin{align}
  x_{1}, \dots, x_{k}, \neg (x_{1} \wedge \dots \wedge x_{k}) \vdash y. \tag{$\beta_{k}$}
\end{align}
  We take $(\beta_0)$ to be $x \vdash x$. Recall that $\height{d}{m}$ is the upset of all elements $a \in \BA{n}$ for $n \assign d + m$ of height strictly greater than $k \geq 0$.

\begin{fact} \label{fact: alpha separation}
  Let $d \geq 2$, $k \geq 0$, and $m \assign k (d-1)$. Then $\pair{\BA{d+m}}{\height{d}{m}}$ satisfies $(\beta_{k})$ but not $(\beta_{k+1})$.
\end{fact}

\begin{proof}
  The condition $x_{i} \in \height{d}{m}$ states that the co-height of $x_{i}$ is at most $d-1$. The~co-height of $x_{1} \wedge \dots \wedge x_{k}$ is thus at most $k (d-1) = m$, or equivalently the height of $\neg (x_{1} \wedge \dots \wedge x_{k})$ is at most $m$. But then the co-height of this element is at least $(d+m) - m = d$ and $\neg (x_{1} \wedge \dots \wedge x_{k}) \notin \height{d}{m}$. The rule $(\beta_{k})$ thus holds in $\pair{\BA{d+m}}{\height{d}{m}}$. On the other hand, let $x_{i}$ for $i \in \range{k+1}$ be joins of some pairwise disjoint sets of atoms of cardinality $d - 1$. Such disjoint sets of atoms exist because $(k+1) (d-1) = (d-1) + k (d-1) < d + k (d-1) = d + m$. Then each $x_{i}$ has height $d - 1$ and $\neg x_{i} \in \height{d}{m}$ for each $x_{i}$. By the disjointness condition, $x_1 \vee \dots \vee x_{k+1}$ has height $(k+1) (d - 1)$, therefore $x_1 \vee \dots \vee x_{k+1} \in \height{d}{m}$. The rule $(\beta_{k+1})$ thus fails in $\pair{\BA{d+m}}{\height{d}{m}}$.
\end{proof}

  Let $\BAclass{n, k}$ denote the subclass of $\BAclass{n}$ axiomatized by~$(\beta_{k})$.

\begin{fact}
  $\BAclass{m,i} \subseteq \BAclass{n,j}$ if and only if either $m \leq n$ and $j \leq i$ or~$n = 1$.
\end{fact}

\begin{proof}
 This follows immediately from Facts~\ref{fact: height separation} and~\ref{fact: alpha separation}.
\end{proof}

\begin{fact}
  $\BAclass{2}$ has infinitely many finitary logical subclasses.
\end{fact}

\begin{theorem}
  Each non-trivial filter subclass of $\BAclass{\infty}$ other than $\BAclass{1}$ contains the structure $\BAm{2} \times \BAm{1}$.
\end{theorem}

\begin{proof}
  Each non-trivial filter subclass $\class{K}$ of $\BAclass{\infty}$ contains $\BAm{1}$, so $\BAclass{1} \subseteq \class{K}$. If $\class{K} \nsubseteq \BAclass{1}$, then the rule $x, y \vdash x \wedge y$ fails in some $\pair{\alg{A}}{F} \in \class{K}$. Restricting to an appropriate substructure, we may take $\alg{A}$ to be $2$-generated. It follows that there is a surjective homomorphism $h\colon \FreeBA{2} \to \alg{A}$ where $\FreeBA{2}$ is the free Boolean algebra over the generators $x, y$. Taking $G \assign h^{-1}[F]$, the structure $\pair{\FreeBA{2}}{G}$ is a homomorphic preimage of $\pair{\alg{A}}{F}$ where $x, y \in G$ but $x \wedge y \notin G$.

  If $\neg x \in G$ or $\neg y \in G$, then $\BAm{2}$ is isomorphic to the substructure of $\pair{\FreeBA{2}}{G}$ generated by $x$ or $y$, therefore $\BAclass{2} \subseteq \class{K}$ and $\BAm{2} \times \BAm{1} \in \class{K}$. We may thus assume that $\neg x, \neg y \notin G$. Let $\pair{\alg{B}}{H}$ be the substructure of $\pair{\FreeBA{2}}{G}$ generated by $\{ x, x \wedge y \}$. If $\neg x \vee \neg y \notin G$, then $\pair{\alg{B}}{H}$ is isomorphic to $\BAm{2} \times \BAm{1}$. On the other hand, if $\neg x \vee \neg y \in G$, then $\pair{\alg{B}}{H}$ is isomorphic to $\pair{\BA{3}}{\height{2}{1}}$, where $\height{2}{1}$ is the upset ($2$-filter) generated by the coatoms of $\BA{3}$. But $\BAm{2} \times \BAm{1}$ embeds into $\pair{\BA{3}}{\height{2}{1}} \times \BAm{1}$, therefore $\BAm{2} \times \BAm{1} \in \class{K}$. (If the coatoms of $\BA{3}$ are denoted $\coatom{1}, \coatom{2}, \coatom{3}$, then this embedding sends the two designated coatoms of $\BAm{2} \times \BAm{1}$ to $\pair{\coatom{1}}{\True}$ and $\pair{\coatom{3}}{\True}$ and the non-designated coatom to $\pair{\coatom{2}}{\False}$.)
\end{proof}

  The filter class generated by the structure $\BAm{m} \times \BAm{n}$ for $m > n$ is axiomatized by the infinite set of rules
\begin{align*}
  \set{\bigwedge \Phi}{\Phi \bsubseteq{n} \Delta_{k}} \vdash \False, \tag{$\gamma_{n, k}$}
\end{align*}
  where $\Delta_{k} \assign \{ x_{1}, \dots, x_{k}, \neg (x_{1} \wedge \dots \wedge x_{k}) \}$ with $k \geq n$. We strongly suspect that this class is not finitely axiomatizable, but we have no proof of this.

\begin{theorem}
  The filter class generated by $\BAm{m} \times \BAm{n}$ for $m > n$ consists precisely of $m$-filters which are either total or are contained in some non-total $n$-filter. It~is axiomatized relative to $\BAclass{m}$ by the implications $(\gamma_{n,k})$ for $k \geq n$.
\end{theorem}

\begin{proof}
  Let $F$ be an $m$-filter where $(\gamma_{n,i})$ holds for each $i$. We show that the $n$-filter $G$ generated by $F$ does not contain $\False$, i.e.\ that $F$ is contained in a non-trivial (prime) $n$-filter. If $G$ did contain $\False$, this would be witnessed by a set $X = \{ x_{1}, \dots, x_{k+1} \} \subseteq F$ such that $\bigwedge Y \in F$ for each $Y \bsubseteq{n} X$ and $\bigwedge X = \False$, hence $x_{k+1} \leq \neg (x_{1} \wedge \dots \wedge x_{k})$. We could therefore assume without loss of generality that $X = \{ x_{1}, \dots, x_{k}, \neg \bigwedge (x_{1} \wedge \dots \wedge x_{k}) \}$. But the rule $(\gamma_{n,k})$ forbids this. Conversely, if $(\gamma_{n,k})$ fails for some $k$, then this provides a set witnessing that the $n$-filter generated by $F$ is trivial. In particular, the upset of $\BAm{m} \times \BAm{n}$ is an $m$-filter contained in the $n$-filter $\BA{m} \times \nonempty{n}$ (the product of the total filter on $\BA{m}$ and the $n$-filter $\nonempty{n} \subseteq \BA{n}$), therefore the rules $(\gamma_{n,k})$ hold in $\BAm{m} \times \BAm{n}$.

  Now let $F$ be an $m$-filter on $\alg{A}$ contained in some non-trivial $n$-filter $G$. It remains to show that $\pair{\alg{A}}{F}$ lies in the filter class generated by $\BAm{m} \times \BAm{n}$. Of course, $G$ extends to a non-trivial prime $n$-filter $H \supseteq G$. The $m$-filter $F$ is the intersection of some prime $m$-filters $F_{i}$ for $i \in I$. This yields an embedding $f\colon \pair{\alg{A}}{F} \into \prod_{i \in I} \BAm{m}$. The prime $n$-filter $H$ corresponds to a strict homomorphism $g\colon \pair{\alg{A}}{H} \to \BAm{n}$, which yields a strict homomorphism $g^{I}\colon \pair{\alg{A}}{H} \to \prod_{i \in I} \BAm{n}$. Taking the product of $f$ and $g^{I}$ yields an embedding of algebras $h\colon \alg{A} \into \prod_{i \in I} (\BA{m} \times \BA{n})$. It remains to show that this is a strict homomorphism $h\colon \pair{\alg{A}}{F} \to \prod_{i \in I} (\BAm{m} \times \BAm{n})$: if $a \in F$, then also $a \in H$, so $f(a)$ belongs to the upset of $\prod_{i \in I} \BAm{m}$, $g(a)$ belongs to the upset of $\BAm{n}$, and $h(a)$ belongs to the upset of $\prod_{i \in I} (\BAm{m} \times \BAm{n})$. On the other hand, if $a \notin F$, then $f(a)$ does not belong to the upset of $\prod_{i \in I} \BAm{m}$, so $h(a)$ does not belong to the upset of $\prod_{i \in I} (\BAm{m} \times \BAm{n})$. Therefore $h$ is a strict homo\-morphism and $\pair{\alg{A}}{F}$ lies in the filter class generated by $\BAm{m} \times \BAm{n}$.
\end{proof}

  The filter class generated by $\BAm{n+1} \times \BAm{n}$ lies strictly between the filter classes generated by $\BAm{n}$ and $\BAm{n+1}$: the upset of $\BAm{n+1} \times \BAm{n}$ is an $(n+1)$-filter but not an $n$-filter, and the upset of $\BAm{n+1}$ is not contained in a non-trivial $n$-filter.

\begin{lemma}
  Let $\pair{\alg{A}}{F}$ be a finite structure in $\BAclass{\infty}$. Then each strict homomorphic image of $\pair{\alg{A}}{F}$ is isomorphic to a substructure of $\pair{\alg{A}}{F}$.
\end{lemma}

\begin{proof}
  For each surjective homomorphism of finite Boolean algebras $r\colon \alg{A} \to \alg{B}$ there is some embedding $s\colon \alg{B} \to \alg{A}$ such that $r \circ s = \idmap_{\alg{A}}$. If $r\colon \pair{\alg{A}}{F} \to \pair{\alg{B}}{G}$ is a strict surjective homomorphism, then $r^{-1}[G] = F$, so $G = (r \circ s)^{-1}[G] = s^{-1}[r^{-1}[G]] = s^{-1}[F]$ , i.e.\ $s\colon \pair{\alg{A}}{F} \to \pair{\alg{B}}{G}$ is a strict embedding.
\end{proof}

\begin{theorem}
  Each finitary filter subclass of $\BAclass{\infty}$ is a logical class.
\end{theorem}

\begin{proof}
  Each finitary filter subclass is generated by its finite structures, since each structure embeds into an ultraproduct of its finitely generated structures. It thus suffices to show that the logical class $\HSinvop \HSop \Sop \Pop (\class{K}) = \HSinvop \HSop(\class{K})$ generated by a filter class $\class{K}$ has the same finite structures as $\class{K}$. But each finite structure in this logical class is a strict homomorphic preimage of a finite structure $\pair{\alg{A}}{F}$ which is a strict homomorphic image of some structure $\pair{\alg{B}}{G}$ in $\class{K}$. Then $\pair{\alg{A}}{F}$ is in fact a strict homomorphic image of some finitely generated (hence finite) substructure $\pair{\alg{C}}{H}$ of $\pair{\alg{B}}{G}$. By the previous lemma, this makes $\pair{\alg{A}}{F}$ isomorphic to a substructure of~$\pair{\alg{B}}{G}$.
\end{proof}

  It remains an open question to determine how many finitary filter sub\-classes (or equivalently, finitary logical subclasses) of $\BAclass{\infty}$ there are.

\subsection*{Data availability} Data sharing not applicable to this article as no datasets were generated or analysed during the current study.

\end{document}